\documentclass{aptpub}

\usepackage{amsfonts}
\usepackage{harvard}
\usepackage{amsmath}

\authornames{F. Daly, C. Lef\`evre and S. Utev} 
\shorttitle{Stein's method and stochastic orderings} 

\begin{document}

\title{Stein's method and stochastic orderings}

\authorone[Universit\"at Z\"urich]{Fraser Daly}
\addressone{Institut f\"ur Mathematik, Universit\"at Z\"urich, Winterthurerstrasse 190, CH-8057 Z\"urich, Switzerland.  Email address: fraser.daly@math.uzh.ch}
\authortwo[Universit\'e Libre de Bruxelles]{Claude Lef\`evre}
\addresstwo{Universit\'e Libre de Bruxelles (U.L.B.), D\'epartement de Math\'ematique, Campus de la Plaine C.P. 210, B-1050 Bruxelles, Belgium.  Email address: clefevre@ulb.ac.be}
\authorthree[University of Nottingham]{Sergey Utev}
\addressthree{University of Nottingham, School of Mathematical Sciences, University Park, NG7 2RD Nottingham, United Kingdom.  Email address: sergey.utev@nottingham.ac.uk}

\begin{abstract}
A stochastic ordering approach is applied with Stein's method for approximation by the equilibrium distribution of a 
birth--death process. The usual stochastic order and the more general $s$-convex orders are discussed. Attention is focused on Poisson and translated Poisson approximation of a sum of dependent Bernoulli random variables, for example $k$--runs in i.i.d. Bernoulli trials.  Other applications include approximation by polynomial birth--death distributions.
\end{abstract}



\section{Introduction} \label{intro}

Stein's method has proved to be an effective tool in probability approximation, 
and has the advantage of being applicable
in the presence of dependence.  See, for example, Stein (1986), and Barbour and Chen (2005) 
for more recent developments.
It is well--known that error bounds obtained via Stein's method may be simplified under 
some assumptions on the dependence
present.  For example, in the presence of negative or positive relation, Stein's method gives simple error bounds in the Poisson
approximation of a sum of indicator random variables.  This is exploited throughout the work of Barbour \emph{et al.} (1992), and will be returned to in our Section \ref{pois}.

In this work, we consider the more general situation of approximation by the equilibrium distribution of a birth--death process,
and examine the situations in which Stein's method leads to simple, easily calculable error bounds.  These error bounds will
typically be differences of moments of our random variables.  As we will see, the assumptions under which we can obtain such error 
bounds are naturally phrased in terms of stochastic orderings.    

Consider a birth-death process on (some subset of) $\mathbb{Z}^+$
with birth rates $\alpha_j$ and death rates $\beta_j$ for $j\geq0$.
Suppose $\beta_0=0$. Let $\pi$ be the stationary distribution of such a
process, with $\pi_j=P(\pi=j)$, $j\geq 0$.  In this work we combine Stein's method with a 
stochastic ordering construction to consider the approximation by $\pi$ of some random variable 
$W$ on $\mathbb{Z}^+$. 

Our random variable $\pi$ satisfies the
identity $E[Ag(\pi)]=0$ for any bounded function $g:\mathbb{Z}^+\mapsto\mathbb{R}$, 
where $A$ is the linear operator defined by
\begin{equation} \label{1} 
Ag(j)=\alpha_j g(j+1)-\beta_j g(j), \quad j\geq 0.
\end{equation}
$A$ is a characterising operator for
$\pi$, in the sense that a random variable $Z=_d\pi$ if and only
if $E[Ag(Z)]=0$ for all $g$ bounded. The construction of such a
characterising operator is the basis of Stein's method for probability
approximation. See the books by Stein (1986), Barbour \emph{et al.}
(1992), Barbour and Chen (2005) and references therein.
For Stein's method applied to birth-death processes, see Brown and Xia (2001) and Holmes (2004).

Given some test function $h$, the so-called Stein equation is defined
by \begin{equation} \label{2}
h(j) - E[h(\pi)] =Af(j), \quad j\geq 0.
\end{equation}
Its solution is denoted $f=f_h=Sh$.  We call $S$ the Stein operator.
Bounds on $S$ are an essential ingredient of Stein's method.

Note that the solution $f$ of the Stein equation depends on the chosen test function 
$h$.  However, for notational convenience in much of the work that follows we will write $f$
rather than $f_h$ or $Sh$.  We will often choose $h(j)=I_{(j\in B)}$ for some $B\subseteq\mathbb{Z}^+$, 
in which case the solution $f$ will depend on the chosen set $B$.

There are several common distributions $\pi$ covered by this framework.
For each of the examples below, bounds are available on the corresponding Stein operator $S$.  
Theorem 2.10 of Brown and Xia (2001) may also be applied to give bounds on $S$
in many cases.
\begin{itemize}
\item If $\alpha_j=\lambda$ and $\beta_j=j$, then $\pi \sim \textrm{Po}(\lambda)$, the Poisson distribution with
mean $\lambda$.  See Barbour \emph{et al.} (1992) and references therein.
\item If $\alpha_j=q(r+j)$ and $\beta_j=j$,
then $\pi \sim \textrm{NB}(r,1-q)$ has a negative binomial distribution. See Brown and Phillips (1999).
\item If $\alpha_j=(n-j)p$ and $\beta_j=(1-p)j$, then $\pi \sim \textrm{Bin}(n,p)$.
See Ehm (1991).
\item In the geometric case, we may, of course,
use the negative binomial operator above. Alternatively we may choose $\alpha_j=q$ 
and $\beta_j=I_{(j\geq1)}$, so that 
$\pi \sim \textrm{Geom}(1-q)$. See Pek\"oz (1996).
\end{itemize}
The present work is organized as follows.  In Section \ref{sect2}, we will derive abstract error 
bounds using Stein's method
combined with some stochastic ordering assumptions in the setting of approximation by
the equilibrium distribution of a birth--death process.  In Section \ref{simple}, 
a simple sufficient condition under which these stochastic ordering
assumptions hold is considered, and some applications are given.  Section \ref{pois} 
discusses Poisson approximation for a sum of
dependent indicators.  We will see how concepts of negative and positive relation relate 
to our stochastic ordering assumptions,
and present generalizations of error bounds derived by Barbour \emph{et al.} (1992).  
Based on this work we move on, in Section \ref{tp}, to consider
translated Poisson approximation.  Applications here will include approximation of 
the number of $k$--runs in i.i.d. Bernoulli trials.
Finally, in Section \ref{bd3}, we give another abstract approximation theorem, 
and consider its application to  a sum of independent 
indicator random variables.    

\section{An abstract approximation theorem} \label{sect2}

Consider Stein's method for approximating the equilibrium distribution of a birth-death process. Our purpose in this section is to derive abstract error bounds under some stochastic ordering assumptions. 

\subsection{A first-order bound} \label{bd1}

Suppose that $W$ is a random variable supported on (some subset of) $\mathbb{Z}^+$
with $\mu_j=P(W=j)$, $j\geq 0$. Set $\mu_{-1}=0$. Our concern is the approximation
of such a variable $W$ by $\pi$, specifically by estimating the difference
$|Eh(W)-Eh(\pi)|$, i.e. $|E[Af(W)]|$. For this, a simple representation of this difference
will be applied with some stochastic ordering assumptions to yield bounds
using Stein's method. We may then bound, for example, the total variation
distance between $\mathcal{L}(W)$ and $\mathcal{L}(\pi)$, defined by
\begin{equation*}
d_{TV}(\mathcal{L}(W),\mathcal{L}(\pi))=\sup_{B\subseteq\mathbb{Z}^+}|P(W\in B)-P(\pi\in B)|.
\end{equation*}
Although we are mainly concerned with approximation in total variation distance, the results we 
derive may also be used with other probability metrics.

Let $\Delta$ be the forward difference operator. Since, with the operator (\ref{1}), 
the choice of $f(0)$ is arbitrary, we follow Brown and Xia (2001) and choose
$f(0)=0$. Writing $f(j)=\Delta f(0)+\cdots+\Delta f(j-1)$,
we thus obtain the representation
\begin{equation}\label{3}
Eh(W)-Eh(\pi) = \sum_{k=0}^\infty\Delta f(k) \sum_{j=k+1}^\infty 
(\alpha_{j-1}\mu_{j-1}-\beta_{j}\mu_{j}).
\end{equation}
In the next subsection, we will extend (\ref{3}) to include the $l$th forward 
differences of $f(\cdot)$, for all $l\geq 1$.

We now consider how this representation may be applied in conjunction with the usual 
stochastic ordering, denoted $\succeq_{st}$. Define two random variables $W_\alpha$ and
$W_\beta$ by 
\begin{equation}\label{4} 
P(W_\alpha=j)=\frac{\alpha_{j-1}
\mu_{j-1}}{E\alpha_W}, \; \mbox{ and } \; 
P(W_\beta=j)=\frac{\beta_{j}\mu_{j}}{E\beta_W}, \quad j\geq 1.
\end{equation}
If $W_\alpha\succeq_{st}W_\beta$ and $E\alpha_W\geq E\beta_W$, we have that
 $\sum_{j=i}^\infty\alpha_{j-1}\mu_{j-1}\geq\sum_{j=i}^\infty\beta_{j}\mu_{j}$ 
for all $i\geq1$. In this case, (\ref{3}) may be bounded to obtain
\begin{equation*}
|Eh(W)-Eh(\pi)| \;\leq\;
\lVert\Delta f\rVert_\infty \, E[\alpha_W (W+1) - \beta_W W].
\end{equation*}
A similar argument holds if we instead assume that $W_\beta\succeq_{st}W_\alpha$
and $E\beta_W\geq E\alpha_W$. We thus obtain the following result.

\begin{proposition} \label{res1}
Assume that one of the two following conditions holds:
\begin{equation}\label{5}
\mbox {either (i)} \;  W_\alpha\succeq_{st} W_\beta \; \mbox { with } 
 E\alpha_W\geq E\beta_W,  \mbox { or (ii)} \; W_\beta\succeq_{st}W_\alpha\; 
\mbox { with } \; E\beta_W\geq E\alpha_W.
\end{equation}
Then,
\begin{equation}\label{6}
|Eh(W)-Eh(\pi)| \;\leq\; \lVert\Delta Sh\rVert_\infty\,
\left|E[\alpha_W (W+1) - \beta_W W]\right|.
\end{equation}
\end{proposition}

\subsection{A $s$-order bound} \label{bd2}

We will now establish our main abstract result. For that, we will have 
recourse to the concept of discrete $s$-convex stochastic ordering, denoted 
$\succeq_{s-cx}$, for any integer $s\geq 1$. See, for example, Lef\`evre 
and Utev (1996) for this notion. Briefly, given any two non-negative 
integer-valued random variables $X$ and $Y$, one says that $X \succeq_{s-cx} Y$ when 
$$
E[f(X)]\leq E[f(Y)] \; \mbox{ for all }\mbox{$s$-convex functions }f, 
$$
that is, for all functions $f$ satisfying $\Delta^s f(j) \geq 0, \; j\geq 0$. 
Note that this ordering implies that $X$ and $Y$ have the same first $s-1$ moments.

To begin with, we introduce a Bernoulli random variable $v_p$ with
$$
P(v_p=1) = p = 1-P(v_p=0),
$$
independently of all other entries. We write $\alpha =E\alpha_W$, $\beta =E\beta_W$,
and in an analogous way to (\ref{4}), we define the random variables $W_\alpha$ and 
$W_\beta$ by
\begin{equation}\label{7} 
P(W_\alpha\in B) = \alpha^{-1} E[\alpha_W I_{(W+1\in B)}],\; \mbox { and } \;
P(W_\beta\in B) = \beta^{-1} E[\beta_W I_{(W\in B)}],
\end{equation}
for any Borel set $B$. For notational convenience, we choose to write 
$C^k_n = {n\choose k}$. 

The key theorem and an immediate corollary will be first stated, the proof of 
the theorem being given after.

\begin{proposition} \label{res3}
Assume that there exists a random variable $Y$ on $\mathbb{Z}^+$ such that 
$W_\beta-Y\geq0$ a.s. and
\begin{equation}\label{8}
W_\alpha \succeq_{s-cx} v_p(W_\beta -Y).
\end{equation} 
Then,
\begin{multline} \label{9}
|Eh(W)-Eh(\pi)| \leq \sum_{t=0}^{s-1}\;|\Delta^t Sh(0)|\; |E(\alpha_W C_{W+1}^{t})
- E(\beta_W C_{W}^{t})|\\
+ \|\Delta^s Sh\|_\infty \; (\alpha E[C_{W_\alpha}^{s}]
-2\alpha p E[C_{W_\beta-Y}^{s}] + (\alpha p+|\alpha p-\beta|) E[C_{W_\beta}^{s}]).
\end{multline}
\end{proposition}

Consider the special case of (\ref{8}) when $p=1$ and $Y=0$ a.s. When $\alpha = \beta$ and under the condition (\ref{11}) below, one has that 
\begin{equation*}
E[\alpha_W (W+1)^{t}] = E[\beta_W W^{t}], \quad t=0, \ldots, s-1,
\end{equation*}
so that the inequality (\ref{9}) reduces to (\ref{12}).
\begin{corollary} \label{res4}
Assume that $\alpha = \beta$, and one of the two following conditions holds:
\begin{equation}\label{11}
\mbox { either (i) } \; W_\alpha\succeq_{s-cx} W_\beta, \; 
\mbox { or (ii) } \; W_\beta\succeq_{s-cx}W_\alpha.
\end{equation}
Then,
\begin{equation}\label{12}
|Eh(W)-Eh(\pi)| \;\leq\; \|\Delta^s Sh\|_\infty \, |E[\alpha_W C_{W+1}^{s}] -
E[\beta_W C_{W}^{s}]|.
\end{equation}
\end{corollary}

We note that Proposition \ref{res1} does not follow as a special case of Corollary \ref{res4},
since this latter result requires the condition $\alpha=\beta$ not needed in Proposition \ref{res1}.

\begin{proof}[Proof of Proposition \ref{res3}]
In the first step we derive a representation 
of $E[Af(W)]$ that generalizes the representation (\ref{3}). Observe 
that (\ref{1}) and (\ref{7}) give
$$
E[Af(W)] =E[\alpha_W f(W+1)] - E[\beta_W f(W)]
= \alpha E[f(W_\alpha)] - \beta E[f(W_\beta)].
$$
Expanding the function $f$ by the discrete Taylor formula, we obtain, for any 
$s=1,2,\ldots$,
$$
f(x) = f(0)+\sum_{k=0}^\infty \Delta f(k)\; I_{(x>k)}
=\sum_{t=0}^{s-1}\Delta^t f(0) \; C_{x}^{t} \;+\;
\sum_{k=0}^\infty \Delta^s f(k) \; C_{x-k-1}^{s-1};
$$
see Lef\`evre and Utev (1996). Thus, we find that 
\begin{eqnarray}\label{13}
E[Af(W)]&=&\sum_{t=0}^{s-1}\Delta^t f(0) \; E[AC_{W}^{t}] +
\sum_{k=0}^\infty \Delta^s f(k) \; E[A C_{W-k-1}^{s-1}]\nonumber\\
&=&\sum_{t=0}^{s-1}\Delta^t f(0)\; (\alpha E[C_{W_\alpha}^{t}]
- \beta E[C_{W_\beta}^{t}])\nonumber\\
&& \; + \; \sum_{k=0}^\infty \Delta^s f(k)\; 
(\alpha E[C_{W_\alpha-k-1}^{s-1}]- \beta E[C_{W_\beta-k-1}^{s-1}]).
\end{eqnarray}

Our next step is to derive an abstract metrics-ordering relationship result,
which is stated below as a separate lemma. Using the bound (\ref{15}) in the 
representation (\ref{13}) then leads to the announced bound (\ref{12}). 
\end{proof}

\begin{lemma} \label{res5}
Let $X$, $Y$ and $Z$ be random variables on $\mathbb{Z}^+$ such that 
\begin{equation}\label{14}
Z-Y \geq 0 \; \mbox { a.s., } \; \mbox { and } \; X\succeq_{s-cx} v_p(Z-Y). 
\end{equation}
Then, for all $a,b \in \mathbb{R^+}$,
\begin{equation}\label{15}
\sum_{k=0}^\infty |a E[C^{s-1}_{X-k-1}] - b E[C^{s-1}_{Z-k-1}]|
\;\leq\; aE[C^s_{X}] - 2ap E[C^s_{Z-Y}]
 + (ap +|ap-b|) E[C^s_{Z}].
\end{equation}
\end{lemma}
\begin{proof}
Letting
$$
w_k^{(s)}(x) =w_{k}(x) = C_{x-k-1}^{s-1},
$$
we get that
\begin{multline} \label{16}
 \sum_{k=0}^\infty |a E(C^{s-1}_{X-k-1}) - b E(C^{s-1}_{Z-k-1})| 
\;\; = \;\; \sum_{k=0}^\infty |a E[w_k(X)] - b E[w_k(Z)]|  \\
\leq a \sum_{k=0}^\infty |E[w_k(X)]-E[w_k(v_p(Z-Y))]| +
a \sum_{k=0}^\infty |E[w_k(v_p Z)]-E[w_k(v_p(Z-Y))]| \\
  + \sum_{k=0}^\infty |a E[w_k(v_p Z)]-b E[w_k(Z)]| \;\; = \;\; S_1+S_2+S_3.
\end{multline}
Let us examine the three sums in (\ref{16}). First, we easily check that
\begin{equation}\label{17}
\sum_{k=0}^\infty E[w_k(Z)]=E[C_{Z}^{s}].
\end{equation}
Using (\ref{17}), we successively find that 
$$
S_3 \;=\; |ap-b| \sum_{k=0}^\infty E[w_k(Z)] 
\;=\; |ap-b| E[C_{Z}^{s}];
$$
since $Z-Y\geq 0$ and $Z\succeq_{st} Z-Y$,
$$
S_2 \;=\; ap \sum_{k=0}^\infty (E[w_k(Z)]-E[w_k(Z-Y)])
\;=\; ap (E[C_{Z}^{s}]-E[C_{Z-Y}^{s}]);
$$
finally, by the assumption (\ref{14}) and a standard property of the order $\succeq_{s-cx}$,
$$
S_1 \;=\; a \sum_{k=0}^\infty [Ew_k(X)- p Ew_k(Z-Y)]
\;=\; a (E[C_{X}^{s}] - pE[C_{Z-Y}^{s}]).
$$
Inserting these three terms in (\ref{16}), we then deduce the bound (\ref{15}).
\end{proof}

\begin{remark}
For $s=p=1$ and $a=b=1$, Lemma \ref{res5} states 
that if $X\succeq_{st} Z-Y \geq 0$, then
an upper bound for the Wasserstein distance between $\mathcal{L}(X)$ and 
$\mathcal{L}(Z)$ is 
\begin{equation}\label{18}
d_{W}(\mathcal{L}(X),\mathcal{L}(Z)) \;=\; \sum_{k=0}^\infty |P(X>k)-P(Z>k)| \;\leq\; 2EY + EX -EZ.
\end{equation}
This bound is of interest in the stochastic ordering context investigated by Kamae 
\emph{et al.} (1977), with random variables on 
$\mathbb{Z}^+$ here. Note that by choosing the optimal coupling $X$, $Z$ and $Y=(Z-X)_+$, 
(\ref{18}) gives the exact bound since 
$$
d_{W}(\mathcal{L}(X),\mathcal{L}(Z)) \;\leq\; 2 E(Z-X)_+  + EX - EZ 
\;=\; E|Z-X| \;=\; d_{W}(\mathcal{L}(X),\mathcal{L}(Z)).
$$
It is worth indicating that an analogous argument allows us to show that 
the same bound (\ref{18}) holds under the single condition
$X+Y\succeq_{st} Z$.
A priori, this result seems to be preferable, since the extra condition 
$Z-Y\geq 0$ is not required. One can see, however, that $X\succeq_{st} Z-Y$ 
does not imply $X+Y\succeq_{st} Z$ in general. As an example, choose $X=U$, 
$Y=U$ and $Z=n$ a.s., where $n$ is any fixed positive integer and $U$ is 
discrete uniform on the set $\{0,1,\ldots,n\}$. Then, $X=U =_d n-U=Z-Y$ so 
that $X\succeq_{st} Z-Y$, but $X+Y=2U$ is not $\succeq_{st}$ than $n=Z$.
\end{remark}

\section{A simple sufficient condition and examples} \label{simple}

In practice, it may be difficult to check directly such conditions as stochastic ordering between $W_\alpha$ and $W_\beta$, as required by (\ref{5}) and (\ref{11}). 
It is thus useful to have available a simple sufficient condition which
we may then apply.

Throughout this subsection, we assume that $\alpha=\beta$ and  
$W_\alpha$ and $W_\beta$ have equal moments of order $t=1,\ldots, s-1$.
That is, we assume
$$
\mbox{condition } (A_s):\quad\quad\quad 
E[\alpha_W (W+1)^t]=E[\beta_W W^t],\hspace{20pt} t=0,\ldots, s-1.$$
A well-known Karlin-Novikoff sufficient condition to guarantee the 
$s$-convex ordering in (\ref{11}) under $(A_s)$ is that our sequence 
$\{\alpha_{j-1}\mu_{j-1} - \beta_j\mu_j\}$ has at most $s$ changes of sign. 

\begin{proposition} \label{prk_n}
Suppose that the condition $(A_s)$ is satisfied and that the 
sequence $\{\alpha_{j-1}\mu_{j-1} - \beta_j\mu_j\}$ has at most $s$ changes of sign. 
Then (\ref{12}) holds.
\end{proposition}
As a consequence, we obtain the following corollary, which extends Proposition A.1 of 
Barbour and Pugliese (2000) to birth-death processes.

\begin{corollary} \label{res2}
Suppose that $E\alpha_W=E\beta_W$.  If the sequence
$\{\alpha_{j-1}\mu_{j-1} - \beta_j\mu_j\}$ is monotone,
then $W_\alpha$ and $W_\beta$ are stochastically ordered,
so that the inequality (\ref{6}) may be applied.
\end{corollary}

We illustrate these results with the following examples.

\begin{example}
Our first example is motivated by Phillips and Weinberg (2000).
Let $W$ have a Bose-Einstein occupancy distribution.  That  is, given
$m, d \geq 1$,
$$ 
\mu_j=P(W=j)=\binom{d+m-j-2}{m-j}\binom{d+m-1}{m}^{-1},
\quad 0\leq j\leq m.
$$
We wish to approximate $W$ by $\pi \sim \textrm{Geom}(p)$ where
$p=(d-1)/(d+m-1)$.  Let $q=1-p$. To obtain our geometric law, we choose 
$\alpha_j=q$ and $\beta_j=I_{(j>0)}$, $j\geq 0$ as birth 
and death rates.

Firstly, one can easily check that in this case, $E\alpha_W=E\beta_W$ and
the sequence $\{q\mu_{j-1}-\mu_j\}$ is non-decreasing,
so that $W_\alpha\succeq_{st}W_\beta$. Using Corollary \ref{res2},
the bound (\ref{6}) then becomes
\begin{equation} \label{19}
|Eh(W)-Eh(\pi)| \;\leq\; p\, \lVert\Delta Sh\rVert_\infty\; |EW-E\pi|.
\end{equation}
Moreover, it is known (see Pek\"oz (1996, Section 2)) that the Stein operator $S$ admits 
here the representation
\begin{equation*} 
Sh(j)=-\sum_{i=j}^\infty\left[h(i)-Eh(\pi)\right]q^{i-j}.
\end{equation*}
From this, we find that $\Delta Sh(k)=-\sum_{i=k}^\infty\Delta h(i)q^{i-k}$, which 
leads to the bound
$$
\lVert\Delta Sh\rVert_\infty\leq\ p^{-1} \lVert\Delta h\rVert_\infty.
$$
Inserting this bound in (\ref{19}) yields the following.
\begin{corollary} \label{cor_ex1}
With $W$ and $\pi$ as above,
\begin{equation*}
|Eh(W)-Eh(\pi)| \;\leq\; \lVert\Delta h\rVert_\infty\, \frac{m}{d(d-1)}.
\end{equation*}
In particular, $d_{TV}(\mathcal{L}(W),\mathcal{L}(\pi))\leq m/d(d-1)$.
\end{corollary}
\end{example}

\begin{example}
Our next examples centre around approximation by so--called 
polynomial birth--death distributions, defined by Brown and Xia (2001) 
as the equilibrium distribution of a birth--death process
with birth and death rates $\alpha_j$ and $\beta_j$ which are polynomial in $j$.  
With such choices, we will write
$\pi \sim \mbox{PBD}(\alpha_j,\beta_j)$.

Suppose that $W$ satisfies $\mu_j=(a+bj^{-1})\mu_{j-1}$
for some $a,b\in\mathbb{R}$.  That is, $W$ belongs to the Katz 
(or Panjer) family of distributions (see Johnson \emph{et al.} (1992, Section 2.3.1)).  
It is well known that in this 
case $W$ must have either either a binomial, Poisson or negative binomial distribution.  

We fix some $l\geq1$ and consider the 
approximation of $W$ by the polynomial birth-death distribution 
$\pi \sim \mbox{PBD}(\alpha,jQ_{l-1}(j))$.  Here we have chosen a 
constant birth rate $\alpha$ and a death rate $\beta_j=jQ_{l-1}(j)$, where 
$Q_{l-1}(j)$ is a non--decreasing, monic polynomial in $j$ 
of degree $l-1$.  This gives us $l$ parameters needed to specify the 
distribution of $\pi$.  We choose these parameters in such a 
way that the condition ($A_l$) is satisfied.

With our choice of birth and death rates we have that
\begin{equation*}
 \alpha\mu_{j-1}-\beta_{j}\mu_j \;=\; 
 \alpha\mu_{j-1}-jQ_{l-1}(j)(a+bj^{-1})\mu_{j-1} \;=\; 
 \mu_{j-1}[\alpha-ajQ_{l-1}(j)-bQ_{l-1}(j)].
\end{equation*}
Noting that $\alpha-ajQ_{l-1}(j)-bQ_{l-1}(j)$ is a polynomial of 
degree $l$ in $j$, and therefore has at most $l$ real roots, we have
that the sequence $\{\alpha_{j-1}\mu_{j-1} - \beta_j\mu_j\}$ has 
at most $l$ changes of sign, so that either 
$W_\alpha\succeq_{l-cx}W_\beta$ or $W_\beta\succeq_{l-cx}W_\alpha$.

Theorem 2.10 of Brown and Xia (2001) gives us that
\begin{equation*}
\sup \lbrace \|\Delta Sh\|_\infty : h(j)=I_{(j\in B)} , 
B\subseteq{\mathbb{Z}^+} \rbrace \leq \alpha^{-1}.
\end{equation*}
Hence, with $h(j)=I_{(j\in B)}$ for some $B\subseteq\mathbb{Z}^+$,
$$
\|\Delta^lSh\|_\infty \;\leq\; 2^{l-1}\|\Delta Sh\|_\infty \;\leq\; 
2^{l-1}\alpha^{-1}.
$$
From Corollary \ref{res4} we thus obtain Corollary \ref{cor_ex2_gen}.
\begin{corollary} \label{cor_ex2_gen}
With $W$ and $\pi$ as above,
\begin{equation} \label{21}
d_{TV}(\mathcal{L}(W),\mathcal{L}(\pi)) \;\leq\; 2^{l-1}\alpha^{-1} 
\left| E\left[\alpha{W+1\choose l}-WQ_{l-1}(W){W\choose l}\right] \right|.
\end{equation}
\end{corollary}

For example, consider the case where $W\sim\mbox{Bin}(n,p)$ and 
$\pi\sim\mbox{PBD}(\alpha, \gamma j+j(j-1))$, so that $l=2$.
Choosing our constants $\alpha$ and $\gamma$ according to 
the prescription above, straightforward calculations give us that
$$
\alpha=n(n-1)p(1-p), \hspace{12pt}\mbox{ and }\hspace{12pt} \gamma=(n-1)(1-2p).
$$
Furthermore,
\begin{multline*}
E[W(W+1)]=np(np+2-p),\hspace{15pt}E[W^2(W-1)]=n(n-1)p^2(np+2-2p),\hspace{15pt}\mbox{and}\\
E[W^2(W-1)^2]=n(n-1)p^2(n^2p^2+4np-5np^2-8p+6p^2+2).
\end{multline*}
Evaluating the bound (\ref{21}) then gives
\begin{corollary} \label{cor_ex2} Assume that 
$W\sim\mbox{Bin}(n,p)$ and 
$\pi\sim\mbox{PBD}(\alpha, \gamma j+j(j-1))$. Then, 
\begin{equation}  \label{22} 
d_{TV}(\mathcal{L}(W),\mathcal{L}(\pi)) \leq 2p^2.
\end{equation}
\end{corollary}

We note that (\ref{21}) does not necessarily give a bound of 
the optimal order.  In the case covered by (\ref{22}), Theorem 3.1 of
Brown and Xia (2001) gives a bound on total variation distance 
of order $O(p^2/\sqrt{\lambda})$, where $\lambda=E[W]=np$.  This 
disparity is due to our rather crude use of the supremum norm in 
obtaining bounds such as (\ref{21}).  In Sections
\ref{tp} and \ref{bd3}, we will consider more refined ways to bound 
the terms of our Stein equation in some particular 
cases when we have two parameters to choose in our approximating 
distribution $\pi$. Despite this disadvantage, we nevertheless 
note that (\ref{21}) gives an explicit bound which may be applied in many contexts.  
\end{example}

\begin{example}
Our final example of this section focuses on mixture 
distributions of the polynomial birth--death type.  Suppose that
$\pi\sim\mbox{PBD}(\alpha,\beta_j)$ and $W\sim\mbox{PBD}(\xi,\beta_j)$, 
for some constant birth rate $\alpha$, polynomial death
rate $\beta_j$ and random variable $\xi$ on $\mathbb{R}^+$.  In this 
case we have that
\begin{equation} \label{23}
\mu_j = \frac{E\left[\mu_0(\xi)\xi^j\right]}{\prod_{k=1}^j\beta_k}, 
\quad j\geq 0.
\end{equation}
We choose $\alpha$ such that $\alpha=E\beta_W$, that is,
$$
\alpha \;=\; E\sum_{j=0}^\infty\beta_j\mu_j \;=\; E\sum_{j=0}^\infty\xi\mu_{j-1} \;=\; E\xi.
$$
Using (\ref{23}), we obtain
\begin{eqnarray*}
\alpha\mu_{j}-\beta_{j+1}\mu_{j+1} &=& E\left[\frac{\mu_0(\xi)\alpha^{j+1}}{\prod_{k=1}^j\beta_j}\left\{\left(\frac{\xi}{\alpha}\right)^j-\left(\frac{\xi}{\alpha}\right)^{j+1}\right\}\right]\\
&=&E\left[\frac{\alpha\mu_0(\xi)}{\mu_0(\alpha)}
\left(1-\frac{\xi}{\alpha}\right)\pi_j\left(\frac{\xi}{\alpha}\right)^j\right].
\end{eqnarray*}
From this, we can see that the sequence $\{\alpha\mu_{j}-\beta_{j+1}\mu_{j+1}\}$ is monotone. 
Hence, Corollary \ref{res2} gives us the following.
\begin{corollary} \label{cor_ex3} 
With $W$ and $\pi$ as above, 
\begin{equation} \label{24}
|Eh(W)-Eh(\pi)| \;\leq\; \lVert\Delta Sh\rVert_\infty
\left|E[\alpha (W+1) - \beta_W W]\right|.
\end{equation}   
\end{corollary}

For example, if $\beta_j=j$ then $W\sim\textrm{Po}(\xi)$ and we take 
$\pi\sim\textrm{Po}(\lambda)$, where $\lambda=E\xi$.
Using the well--known bound on the Stein operator $S$ in this case, namely
\begin{equation} \label{25}
\|\Delta Sh\|_\infty \leq \; \lambda^{-1} (1-e^{-\lambda}) \; \|h\|_\infty,
\end{equation}
evaluating (\ref{24}) gives, after some straightforward calculation,
$$
d_{TV}(\mathcal{L}(W),\mbox{Po}(\lambda)) \;\leq\; \lambda^{-1} (1-e^{-\lambda})\mbox{Var}(\xi), 
$$
a bound that has also been obtained by Barbour \emph{et al.} (1992, Theorem 1.C).
\end{example}

\section{Poisson approximation for a sum of indicators} \label{pois}

Throughout this section, the random variable $W$ of interest is a sum of indicators:
$$  
W=X_1+\cdots+X_n,   
$$ 
where the $X_i$ are Bernoulli variables, possibly dependent, with 
$$
p_i=P(X_i=1)=1-P(X_i=0), \quad 1\leq i \leq n.
$$ 
Using Propositions \ref{res1} and \ref{res3}, we are going 
to investigate the approximation of the sum $W$ by a Poisson 
random variable $\pi \sim \textrm{Po}(\lambda)$. 

Recall that our Poisson variable is derived from (\ref{1}) when 
$\alpha_j=\lambda$ and $\beta_j=j$, so that by (\ref{7}),
\begin{equation}\label{26} 
W_{\alpha}=W+1, \; \mbox { and } \; P(W_{\beta} \in B)=\frac{E[WI_{(W\in B)}]}{EW},
\end{equation}
for any Borel set $B$. In the analysis, an important role will be 
played by the variables 
\begin{equation*}
W_i=W-X_i, \quad 1\leq i \leq n.   
\end{equation*}

\subsection{Total dependence} \label{total}

Firstly, we consider the case where the indicators $X_i$ are 
totally negatively dependent 
in the sense of Papadatos and Papathanasiou (2002). Let us recall that 
$n$ random variables $X_i,\; 1\leq i \leq n$, are totally negatively
dependent (TND) if
\begin{equation}\label{28} 
\textrm{Cov}[g_1(X_i),g_2(W_i)]\leq 0, \quad 1\leq i \leq n,
\end{equation}
for all non-decreasing functions $g_1$, $g_2$ such that the covariance exists.

Papadatos and Papathanasiou (2002, Theorem 3.1) show that the class of
TND indicators includes the standard class of negatively related indicators.
Stein's method for Poisson approximation of a sum of negatively related indicators is discussed 
by, for example, Barbour \emph{et al.} (1992) and Erhardsson (2005).
Recall that indicator random variables $X_1,\ldots,X_n$ are said to be negatively related if
\begin{multline} \label{29}
E[g(X_1,\ldots,X_{i-1},X_{i+1},\ldots,X_n)|X_i=1] \;\leq\;
E[g(X_1,\ldots,X_{i-1},X_{i+1},\ldots,X_n)],\\ 1\leq i \leq n,
\end{multline}
for all non-decreasing functions $g:\lbrace0,1\rbrace^{n-1} \mapsto \lbrace0,1\rbrace$. 

We wish to bound the total variation distance between $\mathcal{L}(W)$ and $\textrm{Po}(\lambda)$. 
For that, we will apply Proposition \ref{res1}. By (\ref{26}), we have that, for any function $g:\mathbb{Z}^+\mapsto\mathbb{R}$, 
$$
Eg(W_\alpha)=Eg(W+1), \; \mbox { and  } \; Eg(W_\beta)=\frac{E[Wg(W)]}{EW}.
$$
Thus, to show that $W_\alpha \succeq_{st} W_\beta$, we must prove that if $g$ is 
non-decreasing, then $EW Eg(W+1)\geq E[Wg(W)]$. In fact, this was 
established by Papadatos and Papathanasiou (2002, Lemma 3.1). 

Using the bound (\ref{25}) on the Stein operator in the Poisson case, 
(\ref{5}) and (\ref{6}) provide the following result.

\begin{theorem} \label{res6}
If the indicators $\{X_i: 1\leq i\leq n\}$ are TND, then $W_\alpha \succeq_{st}W_\beta$. 
If, in addition, $EW\geq\lambda$, then
\begin{equation*}
d_{TV}(\mathcal{L}(W),\textrm{Po}(\lambda)) \leq \frac{1-e^{-\lambda}}{\lambda}\;
([\lambda +1]EW - E[W^2]).
\end{equation*}
\end{theorem}
Further results on, and examples of, TND indicator random variables 
can be found in Papadatos and Papathanasiou (2002). 

Let us now consider the case where the indicators $X_i$ are positively dependent 
in a certain sense. We adapt the definition (\ref{28}) and say that $n$ random 
variables $X_1\ldots,X_n$, are totally positively dependent (TPD) if
\begin{equation*}
\textrm{Cov}[g_1(X_i),g_2(W_i)]\geq 0, \quad 1\leq i \leq n,
\end{equation*}
for all non-decreasing functions $g_1$, $g_2$ such that the covariance exists. 

Association or positive relation is sufficient for TPD.  
This can be established analogously to 
the proof of Theorem 3.1 of Papadatos and Papathanasiou (2002). 
Recall that our indicator random variables are
said to be positively related if (\ref{29}) holds with the 
inequality reversed for all
non-decreasing functions $g:\lbrace0,1\rbrace^{n-1} \mapsto \lbrace0,1\rbrace$.
This standard property is used with Stein's method by, for example, 
Barbour \emph{et al.} (1992) and Erhardsson (2005).

In the sequel, it is assumed that $EW=\lambda$. To get a bound for 
the total variation 
distance, we will apply Proposition \ref{res3}, using the lemma 
stated below. To begin with, 
we introduce a random variable $X_V$, a mixing of our $n$ 
indicators, in which 
the index $V$ is a random variable of law
\begin{equation}\label{32} 
P(V=i)=\frac{EX_i}{\lambda}, \quad 1\leq i \leq n.
\end{equation}

\begin{lemma} \label{res7}
If $EW=\lambda$ and the indicators $\{X_i : 1\leq i\leq n\}$ are TPD, then 
\begin{equation} \label{33} 
W_\beta\succeq_{st}W_\alpha -X_V,
\end{equation}
where $W_\alpha -X_V\geq 0$ a.s. 
\end{lemma}
\begin{proof}
As seen in (\ref{26}), $W_\alpha = W+1$ and thus, 
$W_\alpha -X_V\geq 0$ a.s. Moreover, $W_\beta$ has the so-called 
$W$-size-biased distribution: see, for example, Goldstein and 
Rinott (1996). $W$ being a sum of indicators, 
it is then known that $W_\beta$ admits the representation 
\begin{equation}\label{34}
W_\beta=\sum_{i\not=V}\hat{X}_i+1,
\end{equation}
where $V$ is a random variable of law (\ref{32}), and if $V=v$,
\begin{equation*} 
\hat{X}_i =_d (X_i|X_v=1), \quad i\neq v.
\end{equation*}
Thus, by (\ref{34}), the ordering (\ref{33}) is equivalent to  
$\sum_{i\not=V}\hat{X}_i \succeq_{st}W-X_V$. To establish this, it is 
enough to prove that
\begin{equation*}
\sum_{i\not=v}\hat{X}_i \succeq_{st}W-X_v, \quad 1\leq v\leq n;
\end{equation*}
see Shaked and Shanthikumar (2007). Now, by (\ref{34}) and the TPD assumption, 
we get, for any real $a \geq 0$,
\begin{eqnarray*}
P(\sum_{i\not=v}\hat{X}_i>a)
&=& P(\sum_{i\not=v}X_i>a|X_v=1)\\
&\geq& P(\sum_{i\not=v}X_i>a) \;=\; P(W-X_v>a),
\end{eqnarray*}
which is the desired result.
\end{proof}

Thanks to Lemma \ref{res7}, we may apply Proposition \ref{res3} with $s=p=1$. Noting that by 
(\ref{32}), 
\begin{equation*}
EX_V \;=\; \sum_{i=1}^n p_i\,P(V=i) \;=\; \frac{1}{\lambda}\,\sum_{i=1}^n p_i^2 ,
\end{equation*}
we then get the following result.

\begin{theorem} \label{res8}
If $EW=\lambda$ and the indicators $\{X_i : 1\leq i\leq n\}$ are TPD, then 
\begin{equation*}
d_{TV}(\mathcal{L}(W),\textrm{Po}(\lambda)) \leq \frac{1-e^{-\lambda}}{\lambda}\;
\left\{E[W^2] +2\sum_{i=1}^np_i^2 - \lambda(\lambda +1)\right\}.
\end{equation*}
\end{theorem}

This bound is obtained (and applied) by Barbour \emph{et al.} (1992, Corollary 2.C.4) under the condition of 
positive relation. See also Erhardsson (2005).

\subsection{Local dependence} \label{local}

Our goal in this part is to combine the previous $s$-convex ordering approach 
with a more flexible property of dependence. More precisely, we first introduce 
a concept of local dependence between a set of $n$ indicators $X_1,\ldots,X_n$. 

Let $\mathcal{F}_s $ be the class of all functions $g: \lbrace0,1\rbrace^{n-1} 
\mapsto \mathbb{R^+}$ that are non-decreasing and $s$-convex with $g(0)=0$. 
We say that the $n$ indicators $X_1\ldots,X_n,$ are 
$(s,{\bf \delta})${\it-locally negatively dependent} 
($(s,{\bf \delta})$-LND) if there exist $n$ non-negative reals 
$\delta_1\ldots,\delta_n$ (of sum $>0$) such that 
\begin{equation}\label{37}
E[X_ig(W_i)] \;\leq\; \delta_i\, E[g(W_i)] \; \mbox { for all functions } g\in \mathcal{F}_s, 
\quad 1\leq i \leq n.
\end{equation}
Similarly, $X_1\ldots,X_n,$ are said to be $(s,{\bf \delta})$-locally 
positively dependent ($(s,{\bf \delta})$-LPD) if
\begin{equation}\label{38}
E[X_ig(W_i)] \;\geq\; \delta_i \, E[g(W_i)] \; \mbox { for all functions } g\in \mathcal{F}_s, 
\quad 1\leq i \leq n.
\end{equation}

Let $\delta:=\delta_1+\ldots+ \delta_n$, and denote
\begin{equation*}
v_i= \delta^{-1}\delta_i, \quad 1\leq i \leq n, \; 
\mbox{ and }\; p=EW/\delta \wedge \delta/EW.
\end{equation*}
We then adopt the notation $v_p$ and $X_V$ of Sections \ref{intro} and \ref{total}.

\begin{lemma} \label{res9}
If the indicators $\{X_i : 1\leq i\leq n\}$ are $(s,{\bf \delta})$-LND, then
\begin{equation}\label{40}
W_{\alpha} \succeq_{s-cx} v_p W_{\beta},
\end{equation}
while if the indicators $\{X_i : 1 \leq i\leq n\}$ are $(s,{\bf \delta})$-LPD, then 
\begin{equation}\label{41}
W_{\beta} \succeq_{s-cx} v_p (W_{\alpha} -X_V).
\end{equation}
\end{lemma}
\begin{proof}
The method of proof is built on ideas in Barbour \emph{at al.} (1992), 
Goldstein and Rinott (1996), Papadatos and Papathanasiou (2002) and Reinert (2005). 
Let $g$ be any function belonging to $\mathcal{F}_s$. As a 
preliminary, we observe that $W\leq W_i+1\leq W+1$ a.s. for each $i=1,\ldots,n$.

Now, consider the case of $(s,{\bf \delta})$-LND. Using (\ref{37}) and the assumption 
that $g$ is non-decreasing, we obtain that
\begin{eqnarray*}
E[Wg(W)] &=&\sum_{i=1}^n E[X_ig(W)]  \;\;=\;\; \sum_{i=1}^n E[X_ig(W_i+1)]
\;\;\leq\;\; \sum_{i=1}^n\delta_i E[g(W_i+1)]\nonumber\\
&\leq& \sum_{i=1}^n\delta_i E[g(W+1)] \;\;=\;\; \delta E[g(W_{\alpha})].
\end{eqnarray*}
As $g(0)=0$, and $EW/\delta \geq p\in (0,1]$, we find from (\ref{43}) that
\begin{eqnarray*}
E[g(W_{\alpha})] &\geq& \frac{E[Wg(W)]}{EW}\; \frac{EW}{\delta}\\
&\geq& pE[g(W_{\beta})] \;\;=\;\; E[g(v_p W_{\beta})],
\end{eqnarray*}
hence the ordering (\ref{40}).

The case of $(s,{\bf \delta})$-LPD is treated similarly. By (\ref{38}) and since $g$ is 
non-decreasing, we get
\begin{eqnarray}\label{43}
E[Wg(W)] &=&\sum_{i=1}^n E[X_i g(W_i+1)]
\;\;\geq\;\; \sum_{i=1}^n\delta_i E[g(W_i+1)]\nonumber\\
&=& \delta\sum_{i=1}^nP(V=i) E[g(W+1-X_i)]
\;\;=\;\; \delta E[g(W_{\alpha}-X_V)].
\end{eqnarray}
As before, we then deduce from (\ref{43}) that
\begin{eqnarray*}
E[g(W_{\beta})]&=& \frac{E[Wg(W)]}{EW}
\;\;\geq\;\; E[g(W_{\alpha}-X_V)]\;\frac{\delta}{EW}\\
&\geq& p E[g(W_{\alpha}-X_V)] \;\;=\;\; E[g(v_p(W_{\alpha}-X_V))],
\end{eqnarray*}
proving the ordering (\ref{41}).
\end{proof}

Combining Proposition \ref{res3} and Lemma \ref{res9} would then 
allow us to derive an upper bound for the total variation distance. 

\subsection{Approximate local dependence} \label{aplocal}

Approximate local dependence is becoming a rather popular topic in probability. 
For works related to this idea, see for example Chen (1975), 
Barbour \emph{et al.} (1992) and Chatterjee \emph{et al.} (2005). We wish now to to derive 
an abstract Poisson approximation theorem by combining stochastic ordering 
with such an approach. 

We say that the $n$ indicators $X_1,\ldots,X_n$ are approximately locally negatively dependent 
(ALND) if there exist 
$n$ non-negative reals $\delta_1,\ldots,\delta_n$ (of sum $\delta>0$), and $n$ 
random variables $Y_1,\ldots,Y_n$ on $\mathbb{Z}^+$ such that 
\begin{equation}\label{44}
E[X_ig(W_i-Y_i)] \;\leq\; \delta_i \, E[g(W_i-Y_i)], \quad 1\leq i \leq n,
\end{equation}
for all non-negative, non-decreasing functions $g$. Similarly, $X_1,\ldots,X_n$ are said to be  approximately locally positively dependent (ALPD) if 
\begin{equation}\label{45}
E[X_ig(W_i-Y_i)] \;\geq\; \delta_i \, E[g(W_i-Y_i)], \quad 1\leq i \leq n,
\end{equation}
for all non-negative, non-decreasing functions $g$.

Define
$$
\varepsilon = \sum_{i=1}^n E[X_iY_i], \; \mbox { and } \;
\varepsilon_\ast = \varepsilon + \sum_{i=1}^n\delta_i E[X_i+Y_i],
$$
and let
\begin{equation*}
c_\lambda = (\lambda+1)(1-e^{-\lambda})/\lambda + 2d_\lambda, \; \mbox 
{ with } \; d_\lambda = 1 \wedge \sqrt{2/e\lambda}. 
\end{equation*}

\begin{theorem} \label{res11}
If $EW=\lambda$ and the indicators $\{X_i : 1\leq i\leq n\}$ are ALND, then 
\begin{equation}\label{47}
d_{TV}(\mathcal{L}(W),\textrm{Po}(\lambda)) \leq \frac{1-e^{-\lambda}}{\lambda}\; 
\left(|\mbox{Var}(W)-\lambda|+ 2\varepsilon\right) + c_\lambda\; |\delta-\lambda|,
\end{equation}
while if the indicators $\{X_i : 1\leq i\leq n\}$ are ALPD, then 
\begin{equation*}
d_{TV}(\mathcal{L}(W),\textrm{Po}(\lambda)) \leq \frac{1-e^{-\lambda}}{\lambda} \;
(|\mbox{Var}(W)-\lambda|+2\varepsilon_\ast)+ c_\lambda \;|\delta-\lambda|.
\end{equation*}
\end{theorem}

Before proving Theorem \ref{res11}, we give an example of its application.

\begin{example}
We examine a variation of the classical birthday problem; 
see also Barbour \emph{et al.} (1992). Suppose we independently colour 
$N\geq2$ points with one of $m$ colours, each colour being chosen equiprobably. 
Let $\Gamma$ be the set of all subsets $i\subseteq\{1,\ldots,N\}$ of size $2$. 
For $i\in \Gamma$, let $Z_i$ be the indicator that the points indexed 
by $i$ have the same colour. Moreover, suppose we choose uniformly $r$ 
of the $|\Gamma|={N\choose 2}$ pairs of points, independently of the
colourings chosen. For $i\in\Gamma$, we let $\xi_i=0$ if the pair of 
points indexed by $i$ is chosen, and otherwise set $\xi_i=1$. 

Set $W=\sum_{i\in\Gamma}Z_i\xi_i$. This counts the number of pairs of points with the same colour, excluding those $r$ pairs of points we have chosen. In the case where 
$r=0$, this corresponds to the classical birthday problem. A bound in the Poisson
approximation of $W$ in this case is given by Arratia \emph{et al.} (1989, Example 2).

We observe that for all $i,j\in\Gamma$, $E\left[Z_i\right]=m^{-1}$ and 
$E\left[Z_iZ_j\right]=m^{-2}$. Furthermore, 
$$
E\left[\xi_i\right]=\frac{{N\choose 2}-r}{{N\choose 2}},\hspace{12pt}
\mbox{ and }\hspace{12pt} E\left[\xi_i\xi_j\right]=
\frac{{N\choose 2}-r}{{N\choose 2}}
\left(\frac{{N\choose 2}-r-1}{{N\choose 2}-1}\right), \quad i\not=j.
$$
Straightforward calculations then give
$$
\lambda=E[W]=\frac{{N\choose 2}-r}{m},\hspace{12pt}
\mbox{ and }\hspace{12pt} \lambda-\mbox{Var}(W)=\frac{{N\choose 2}-r}{m^2}.
$$
Now, we write $W_i=W-Z_i\xi_i\;$ and choose
$$
Y_i=\sum_{j\not=i}Z_j\xi_jI_{(i\cap j\not=\emptyset)}, \;\;\mbox{ and }\;\; \delta_i=E\left[Z_i\xi_i\right].
$$
The condition (\ref{44}) holds true with these choices. Indeed, 
$W_i-Y_i$ is independent of $Z_i$ and the $\xi_i$ are negatively 
related by construction. Thus, for all non--decreasing functions $g$, we have
$$
E\left[Z_i\xi_ig(W_i-Y_i)\right] \;\;=\;\; 
E\left[Z_i\xi_i\right]E\left[g(W_i-Y_i)|\xi_i=1\right] \;\;\leq\;\; 
E\left[Z_i\xi_i\right]E\left[g(W_i-Y_i)\right],
$$
as required. We further see that
\begin{eqnarray*}
\varepsilon \;=\; \sum_{i\in\Gamma}E[Z_i\xi_iY_i] &=& 
\sum_{i\in\Gamma}\sum_{j\not=i}E\left[Z_iZ_j\right]
E\left[\xi_i\xi_j\right]I_{(i\cap j\not=\emptyset)}\\
&=& \frac{2(N-1)\left\{{N\choose 2}-r\right\}
\left\{{N\choose 2}-r-1\right\}}{m^2\left\{{N\choose 2}-1\right\}}.
\end{eqnarray*}
Evaluating (\ref{47}) then gives the following bound.
\begin{corollary}\label{cor_ex4}
With $W$ as above,
\begin{equation*}
d_{TV}(\mathcal{L}(W),\textrm{Po}(\lambda)) \leq \frac{1-e^{-\lambda}}{m}\left\{1+4(N-1)\left(\frac{{N\choose 2}-r-1}{{N\choose 2}-1}\right)\right\}.
\end{equation*}
\end{corollary}
In the case $r=0$, a bound of the same order was established by Arratia \emph{et al.} (1989, Example 2).
\end{example}

\subsection{Proof of Theorem \ref{res11}} \label{proof1}

(i) Consider the ALND case. We suppose first that $f$ 
is any non-negative, non-decreasing function. Arguing as for 
Lemma \ref{res9}, we have
\begin{eqnarray*}
E[Wf(W)] &=&\sum_{i=1}^n E[X_if(W)] \;\;=\;\; \sum_{i=1}^n E[X_if(W_i+1)]\\
&=& \sum_{i=1}^n E[X_if(W_i-Y_i+1)] +\sum_{i=1}^n E\{X_i[f(W_i+1)-f(W_i-Y_i+1)]\},
\end{eqnarray*}
which we denote by $T_1+T_2$. We bound the sum $T_2$ by noting that 
$$
|f(x)-f(y)| \;\leq\; \|\Delta f\|_\infty\; |x-y|, 
$$
which yields
$$
T_2 \;\leq\; \|\Delta f\|_{\infty} \sum_{i=1}^n E(X_i Y_i) \;=\; \|\Delta f\|_{\infty} 
\; \varepsilon.
$$
For the sum $T_1$, by (\ref{44}) and since $f$ is non-decreasing, we get
$$
T_1 \;\leq\; \sum_{i=1}^n\delta_i E[f(W_i-Y_i+1)] \;\leq\; \sum_{i=1}^n \delta_i E[f(W+1)]
\;=\; \delta E[f(W+1)].
$$
Inserting these two bounds, we find that
\begin{eqnarray}\label{50}
E[Af(W)]&=&\lambda E[f(W+1)] - E[Wf(W)] \nonumber\\
 &\geq & -(\delta-\lambda) E[f(W+1)] - \|\Delta f\|_\infty \; \varepsilon.
\end{eqnarray}
To get an upper bound, we define a function $\tilde{f}$ on 
$\{0, 1, \ldots, n-1\}$ by
\begin{equation}\label{51}
\tilde{f}(x)=\|f\|_\infty  + \|\Delta f\|_\infty\; x - f(x).
\end{equation}
Note that $\tilde{f}$ is, as $f$, a non-negative, non-decreasing function.
By assumption, $EW=\lambda$ so that $E[A1]=0$; observe also that 
$E[AW]=\lambda E[W+1]- E[W^2] = -[\mbox{Var}(W)-\lambda]$. Thus,
\begin{eqnarray*}
E[A\tilde{f}(W)]&=&\|g\|_\infty \; E[A1] + \|\Delta f\|_\infty \; E[AW]- E[Af(W)]\\
&=& -\|\Delta f\|_\infty \; [\mbox{Var}(W)-\lambda]- E[Af(W)].
\end{eqnarray*}
On the other hand, (\ref{50}) is applicable to the function $\tilde{f}$, so that
$$
E[A \tilde{f}(W)] \geq
 -(\delta-\lambda) E[\tilde{f}(W+1)] -\|\Delta \tilde{f}\|_\infty \; \varepsilon.
$$
From these two formulas, we deduce that
\begin{equation}\label{52}
E[Af(W)] \leq \|\Delta \tilde{f}\|_\infty \; \varepsilon
+ (\delta-\lambda) E[\tilde{f}(W+1)]  + \|\Delta f\|_\infty \; |\mbox{Var}(W)-\lambda|.
\end{equation}

Now, let $f$ be an arbitrary function. We start with the standard decomposition
$f=f_+ - f_-$,
where $f_+$ and $f_-$ are non-negative, non-decreasing functions with, of course,
\begin{equation}\label{53}
\|\Delta^j f_+\|_\infty\leq \|\Delta^j f\|_\infty, \; \mbox { and }\; 
\|\Delta^j f_-\|_\infty\leq \|\Delta^j f\|_\infty, \quad j=0,1.
\end{equation}
By (\ref{50}) and (\ref{52}), we obtain an upper bound
\begin{eqnarray*}
E[Af(W)]&=& E[Af_+(W)] - E[Af_-(W)]\\
&\leq& \|\Delta \tilde{f_+}\|_{\infty} \;\varepsilon
+ (\delta-\lambda) E[\tilde{f_+}(W+1)]  + \|\Delta f_+\|_\infty \; |\mbox{Var}(W)-\lambda|\\
&&\qquad
+(\delta-\lambda) E[f_-(W+1)] +\|\Delta f_-\|_\infty \; \varepsilon\\
&=& \|\Delta f_+\|_\infty \; |\mbox{Var}(W)-\lambda| + 
(\|\Delta \tilde{f_+}\|_{\infty}+\|\Delta f_-\|_\infty) \;\varepsilon\\
&&\qquad 
+ (\delta-\lambda) \; \{\|f_+\|_\infty + \|\Delta f_+\|_\infty \;(\lambda +1)
 -E[f(W+1)]\},
\end{eqnarray*}
using (\ref{51}) and $EW=\lambda$ for the last equality. By a similar method, we find 
as a lower bound 
\begin{eqnarray*}
E[Af(W)] &\geq& - (\delta-\lambda) E[{f_+}(W+1)] - \|\Delta f_+\|_\infty \; 
\varepsilon\\ 
&& \qquad - \|\Delta \tilde{f_-}\|_\infty\; \varepsilon
-(\delta-\lambda) E[\tilde{f_-}(W+1)] -\|\Delta f_-\|_\infty \; |\mbox{Var}(W)-\lambda|\\
&=& -\|\Delta f_-\|_\infty \; |\mbox{Var}(W)-\lambda| - 
(\|\Delta f_+\|_\infty + \|\Delta \tilde{f_-}\|_\infty)\; \varepsilon\\
&&\qquad -(\delta-\lambda) \; \{\|f_-\|_\infty + \|\Delta f_-\|_\infty \;(\lambda +1)
 +E[f(W+1)]\}.
\end{eqnarray*}
By (\ref{53}) and since $\|\Delta \tilde{f}\|_\infty \leq \|\Delta f\|_\infty$, 
combining the two previous bounds then yields
\begin{equation}\label{54}
|E[Af(W)]|\leq \|\Delta f\|_\infty \; (|\mbox{Var}(W)-\lambda| +2 \varepsilon)
+ |\delta-\lambda| \;[2 \|f\|_\infty +  \|\Delta f\|_\infty \;(\lambda +1)].
\end {equation}
With $f=Sh$, it now suffices to apply in (\ref{54}) the standard bounds 
$$
\|\Delta Sh\|_\infty \leq 
\lambda^{-1}(1-e^{-\lambda}) \|h\|_\infty, \mbox { and } \; 
\|Sh\|_\infty\leq d_\lambda \|h\|_\infty,
$$
which gives (\ref{47}).

(ii) The ALPD case is dealt with analogously. For $f$ non-negative, 
non-decreasing, we first write that
\begin{eqnarray*}
E[Wf(W)] & = &\sum_{i=1}^n E[X_if(W_i-Y_i+1)]
+\sum_{i=1}^n E[X_i\{f(W_i+1)-f(W_i-Y_i+1)\}]\\
&\geq &\sum_{i=1}^n E[X_if(W_i-Y_i+1)]  - \|\Delta f\|_\infty \; \varepsilon.
\end{eqnarray*}
By (\ref{45}), we then get that
\begin{eqnarray*}
E[Wf(W)] &\geq& \sum_{i=1}^n\delta_i E[f(W_i-Y_i+1)]- \|\Delta f\|_\infty
\;\varepsilon\\
&=& \delta E[f(W+1)]
- \sum_{i=1}^n\delta_i E[f(W+1)-f(W_i-Y_i+1)]
- \|\Delta f\|_\infty \;\varepsilon\\
&\geq& \delta E[f(W+1)]
- \|\Delta f\|_\infty \; \sum_{i=1}^n\delta_i E(X_i+Y_i)
- \|\Delta f\|_\infty \; \varepsilon \\
&=& \delta E[f(W+1)] - \|\Delta f\|_\infty \;\varepsilon_\ast.
\end{eqnarray*}
Overall, we find that
\begin{equation*}
E[Af(W)] = \lambda E[f(W+1)] - E[Wf(W)]
\geq  -(\delta-\lambda) E[f(W+1)] +\|\Delta f\|_\infty\; \varepsilon_\ast.
\end{equation*}
The rest of the proof follows as in the ALND case.

\section{Translated Poisson approximation} \label{tp}

We assume, as in Section \ref{pois}, 
that $W=X_1+\cdots+X_n$ is a sum of (possibly dependent) 
indicator random variables, with $p_i=P(X_i=1)$.  Denote 
$$
\lambda_k=\sum_{i=1}^np_i^k, \quad \lambda=\lambda_1=E[W], 
\quad \mbox {and}\quad \sigma^2=\mbox{Var}(W).
$$
We are going to discuss the approximation of $W$ by a 
translated Poisson distribution.

\subsection{Main results}

A random variable $Z$ has a translated Poisson distribution 
$\textrm{TP}(\lambda,\sigma^2)$ if $Z$ is distributed 
as $Z^\prime+\rho$, where 
$Z^\prime\sim\mbox{Po}(\sigma^2+\gamma)$ with  
$$
\rho=\lambda-\sigma^2-\gamma, \;\;\mbox{  and  }\;\; 
\gamma=\langle\lambda-\sigma^2\rangle\in[0,1),
$$
$\langle x\rangle=x-\lfloor x\rfloor$ denoting the fractional part of $x$.

We note that $E[Z]=\lambda$ and $\sigma^2\leq\mbox{Var}(Z)=
\sigma^2+\gamma<\sigma^2+1$, so that our approximating translated Poisson
distribution has a mean equal to, and variance close to, that of $W$.  
We would thus expect a closer approximation than could be
obtained by simply using the one--parameter Poisson distribution.  
The variances of $W$ and $Z$ cannot necessarily be made to match
exactly, as we must shift our Poisson distribution by an integer.  
However, the error term arising from this mismatch does not 
adversely affect the order of the bounds we obtain, as we shall see below.  

The following results give us bounds in translated 
Poisson approximation for $W$ under some stochastic ordering assumptions.
We defer the proofs of Theorems \ref{tpneg} and \ref{tppos} 
until Section \ref{tpproofs}, giving first some examples of their 
application, in Section \ref{runs}.

Our bounds demonstrate convergence to a translated Poisson 
distribution if $\sigma\rightarrow\infty$ as $n\rightarrow\infty$.
Bounds on the total variation distance between $\mathcal{L}(W)$ 
and a translated Poisson random variable may still be found if this 
is not the case,
but require a different analysis of the error terms.  For example, 
in proving Theorems \ref{tpneg} and \ref{tppos}, we write
$P(W-\rho<0)\leq\sigma^{-2}$.  This error term may be reduced, or 
even omitted altogether depending on the problem at hand, with a 
more careful analysis.  This could give us good bounds in cases 
where $\sigma\rightarrow\sigma_\infty<\infty$ as $n\rightarrow\infty$.  

In the sequel, we let $W^s$ be a random variable having the 
$W$--size--biased distribution, and $v_q$ be an indicator random 
variable, independent of all else, with $P(v_q=1)=q$. As before,  
we write $W_i=W-X_i$, $1\leq i\leq n$, and for any random index $V$ we let
$W_V=W-X_V$.

\begin{theorem} \label{tpneg}
Suppose that $X_1,\ldots,X_n$ are negatively related, and there is 
$q\in[0,1]$ and $l\in\mathbb{Z}^+$ such that
\begin{equation}\label{56}
(W+1|X_k=0)\preceq_{st}(W+l+v_q|X_k=1),\quad 1\leq k\leq n.
\end{equation}
Then,
\begin{multline} \label{57}
d_{TV}(\mathcal{L}(W),\mbox{TP}(\lambda,\sigma^2)) \leq 
\frac{2}{\sigma^2} + \frac{\lambda_2+(l+q)(\lambda-\lambda_2)}{\lambda\sigma} \\
+ \frac{l(l+2q-1)(\lambda-\lambda_2)}{\sigma^2}\,d_{TV}(\mathcal{L}
(W^s),\mathcal{L}(W^s+1)).
\end{multline}
\end{theorem}

\begin{theorem} \label{tppos}
Suppose that $X_1,\ldots,X_n$ are positively related, and there is 
$q\in[0,1]$ and $l\in\mathbb{Z}^+$ such that
\begin{equation}\label{58}
(W+1|X_k=0)\succeq_{st}(W-l-v_q|X_k=1),\quad 1\leq k\leq n.
\end{equation}
Then,
\begin{multline} \label{59}
d_{TV}(\mathcal{L}(W),\mbox{TP}(\lambda,\sigma^2)) \leq 
\frac{2}{\sigma^2} + \frac{\lambda_2+(l+q)(\lambda-\lambda_2)}{\lambda\sigma} \\
+ \frac{(l+1)(l+2q)(\lambda-\lambda_2)}{\sigma^2}\,d_{TV}(\mathcal{L}(W^s),
\mathcal{L}(W^s+1)).
\end{multline}
\end{theorem}

Consider the stochastic ordering assumptions (\ref{56}) and (\ref{58}). 
We note that the choice of $l$ and $q$ is not unique, in that choosing 
$l=m$, $q=1$ gives the same assumption as choosing $l=m+1$, $q=0$. 
It is easily checked, however, that each of these choices gives rise to 
the same  bounds in (\ref{57}) and (\ref{59}). In the examples below, 
we will verify the validity of such stochastic orderings by using 
an appropriate coupling argument.

\subsection{Applications} \label{runs}

\begin{example}
Suppose that $X_1,\ldots,X_n$ are independent. Thus, 
they are also negatively related. Moreover, the condition (\ref{56}) 
is true for $q=l=0$. Therefore, (\ref{57}) is applicable and yields the following.
\begin{corollary}\label{cor_ex5}
With $W$ as above,
$$
d_{TV}(\mathcal{L}(W),\mbox{TP}(\lambda,\sigma^2)) \leq \frac{\lambda_2}{\lambda\sigma} + \frac{2}{\sigma^2}.
$$
\end{corollary}
This bound is of the order we would expect: see also 
\v Cekanavi\v cius and Va\v \i tkus (2001).
\end{example}

\begin{example}
Suppose that $m$ balls are placed into $N$ urns, in 
such a way that no urn contains more than 
one ball and all arrangements are equally likely.  Let $W$ be the number 
of balls in the first $n$ urns. Thus, $W$ has a hypergeometric distribution with
$$
\lambda=\frac{mn}{N},\;\;\mbox{  and  }\;\;\sigma^2=\frac{mn(N-m)(N-n)}{(N-1)N^2}.
$$   

We set $X_i$ to be the indicator that the $i$th urn contains a ball, so 
that $W=X_1+\cdots+X_n$.  By construction, these indicators are
negatively related. The condition (\ref{56}) holds for $q=1$ and $l=0$. 
To see this, we construct $(W+1|X_k=0)$ by 
considering the $N$ urns and excluding the $k$th.  Distribute the $m$ 
balls in these $N-1$ urns, such that all arrangements are equally 
likely, and count
the number of the first $n$ urns that are occupied.  Adding one to this 
count gives us our random variable.  
We then choose (uniformly and independently of what has gone before) 
one of the occupied urns.  Take the ball from this urn and place 
it in urn $k$.  This gives us $(W+1|X_k=1)$.  If the ball chosen is 
from one of the first $n$ urns, the number of occupied urns is the
same as before.  Otherwise, we have increased the number of 
occupied urns within the first 
$n$. Evaluating the bound (\ref{57}) then gives Corollary \ref{cor_ex6}.
\begin{corollary}\label{cor_ex6}
For $W$ having our hypergeometric distribution,
$$
d_{TV}(\mathcal{L}(W),\mbox{TP}(\lambda,\sigma^2)) \;\leq\; 
\frac{1}{\sigma}+\frac{2}{\sigma^2} \;=\; \sqrt{\frac{N^2(N-1)}{mn(N-m)(N-n)}} 
+ \frac{2N^2(N-1)}{mn(N-m)(N-n)}.
$$
\end{corollary}
R\"ollin (2007, Section 4.1) has considered translated Poisson approximation for the 
hypergeometric distribution, and shows that if 
$m=O(n)$  and $N=O(n)$, then one gets a bound in total variation 
distance of  order $O(1/\sqrt{n})$.  This order is also
reflected in our result.
\end{example}

\begin{example}
Suppose $\xi_1,\ldots,\xi_n$ are i.i.d. Bernoulli random variables with
\begin{equation*}
p=P(\xi_i=1)=1-P(\xi_i=0), \quad 1\leq i \leq n.
\end{equation*} 
Fix an integer $k\geq2$, and define 
$$
X_i=\xi_{i}\xi_{i+1}\cdots\xi_{i+k-1}, \quad \mbox {and} \quad W=\sum_{i=1}^nX_i, 
$$
in which, to avoid edge effects, all indices are treated modulo $n$. 
Thus, $W$ counts the number of $k$--runs in our Bernoulli trials. Observe that
$$
\lambda= np^k, \quad \lambda_2=np^{2k}, \quad \mbox{and}\quad 
\sigma^2=\frac{np^k}{1-p}\;(1+p-p^k[2+(2k-1)(1-p)]).
$$
Translated Poisson approximation for $k$--runs was treated by 
R\"ollin (2005, Section 3.2), who gives a bound in total variation distance of 
the form $K/\sqrt{n}$, for some constant $K=K(k,p)$ independent of $n$.  
Barbour and Xia (1999, Section 5) also give a bound of this order for
2--runs.  We shall use our Theorem \ref{tppos} to give an explicit 
bound with this same order.  

It is easily seen that the variables $X_1,\ldots,X_n$ are positively 
related. The condition (\ref{58}) holds by choosing $q=1$ and $l=2k-3$.  
To see that, consider the following construction.  Given the 
Bernoulli random variables $\xi_1,\ldots,\xi_n$, fix some $m\leq n$ and set  
$\xi_{m}=\xi_{m+1}=\cdots=\xi_{m+k-1}=1$, while the others remain 
independent Bernoulli random variables with parameter $p$. 
Counting the number of $k$--runs in these $n$ Bernoulli trials 
gives us $(W|X_m=1)$.  Suppose now we resample
the random variables $\xi_m,\ldots,\xi_{m+k-1}$, conditional on at 
least one of these being zero.  Counting the
number of $k$--runs now gives us $(W|X_m=0)$.  In this resampling 
procedure, one can remove at most $2k-1$ of the $k$--runs that were
originally present.  Thus, our construction implies that 
$(W|X_m=0) + 2k-1 \geq (W|X_m=1)$, or, equivalently, 
$(W+1|X_m=0)\geq(W-2k+2|X_m=1)$, hence the announced values of $q$ and $l$.

Following the work of Section \ref{pois}, to construct $W^s$ we 
choose an index $V$ uniformly from $\{1,\ldots,n\}$, and set
$\xi_V=\xi_{V+1}=\cdots=\xi_{V+k-1}=1$, while the other $\xi_i$ 
remain independent Bernoulli random variables with parameter $p$.  
Lemma 2.1 of Wang and Xia (2008) thus gives us that
$$
d_{TV}(\mathcal{L}(W^s),\mathcal{L}(W^s+1)) \leq 1 \wedge \frac{2.3}{\sqrt{(n-k-1)p^k(1-p)^3}}.
$$  
Using this, Theorem \ref{tppos} yields the following.
\begin{corollary} \label{runsprop}
Let $W$ count the number of $k$--runs in $n$ independent 
Bernoulli trials, each with success probability $p$.  Then,
\begin{multline} \label{60}
d_{TV}(\mathcal{L}(W),\mbox{TP}(\lambda,\sigma^2)) 
\leq \frac{2}{\sigma^2} + \frac{p^k+(2k-2)(1-p^k)}{\sigma}\\ 
+ \frac{(2k-2)(2k-1)np^k(1-p^k)}{\sigma^2}
\left(1\wedge\frac{2.3}{\sqrt{(n-k-1)p^k(1-p)^3}}\right).
\end{multline}
\end{corollary}

Our bound (\ref{60}) has the same order as that of R\"ollin (2005, Theorem 5) 
and Barbour and Xia (1999, Theorem 5.2) (this latter result applying only to the 
2--runs case).  Numerical comparison of the bounds shows that ours 
generally performs well compared to these other bounds, 
often giving a better result.  Table 1 gives some illustrations, 
with values for comparison taken from R\"ollin (2005).   
\begin{table} 
\begin{center}
\caption{Numerical comparisons for 2--runs.  Upper bounds on total 
variation distance from (a) our result (\ref{60}), (b) R\"ollin 
(2005) and (c) Barbour and Xia (1999).  Missing values are due 
to restrictions on choice of parameters.}
\begin{tabular}{ccccccc}
\hline
&& $p=0.10$ & $p=0.25$ & $p=0.50$ & $p=0.75$ & $p=0.90$\\
\hline
& (a) & 0.1553 & 0.0675 & 0.0500 & 0.0814 & 0.2512 \\
$n=10^6$ & (b) & 0.4463 & 0.2334 & 0.1747 & 0.5528 & $>1$ \\
& (c) & 0.0304 & -- & 0.1251 & 0.6014 & -- \\ 
\hline
&(a) & 0.0155 & 0.0067 & 0.0050 & 0.0081 & 0.0251 \\
$n=10^8$ & (b) & 0.0445 & 0.0233 & 0.0175 & 0.0553 & 0.2554 \\
& (c) & 0.0030 & -- & 0.0125 & 0.0601 & -- \\ 
\hline 
&(a) & 0.0016 & 0.0007 & 0.0005 & 0.0008 & 0.0025 \\
$n=10^{10}$ & (b) & 0.0045 & 0.0023 & 0.0017 & 0.0055 & 0.0255 \\
& (c) & 0.0003 & -- & 0.0013 & 0.0060 & -- \\ 
\hline
\end{tabular}
\end{center}
\end{table}
\end{example}

\subsection{Proof of Theorems \ref{tpneg} and \ref{tppos}} \label{tpproofs}

Our proof is based on that of Propositions \ref{res1} and \ref{res3}, 
using the characterising operator for the Poisson distribution.
We find representations of our Stein equation in conjunction with 
which our dependence and stochastic ordering assumptions may be 
applied.

Throughout this section we let $f=Sh$ be the solution to the 
Stein equation (\ref{2}) with the choices $\alpha_j=\sigma^2+\gamma$ and
$\beta_j=j$, corresponding to the Poisson distribution with mean 
$\sigma^2+\gamma$.  We suppose the test function $h$ has the form
$h(j)=I_{(j\in B)}$ for some $B\subseteq\mathbb{Z}^+$.  We write 
$g_B(j)=f(j-\rho)$.  We note that $g_B$ depends on the choice of 
set $B$, though for notational convenience we will often write simply
$g$ for $g_B$.  We note further that bounds on 
the supremum norm of $f$ also apply to $g$, so that in particular 
$\|\Delta g_B\|_\infty\leq\sigma^{-2}$ for each $B\subseteq\mathbb{Z}^+$.

Following R\"ollin (2007, Section 3), we obtain from the Stein equation that
\begin{equation} \label{61}
d_{TV}(\mathcal{L}(W),TP(\lambda,\sigma^2))\leq \sup_{B\subseteq\mathbb{Z}^+} |E[(\sigma^2+\gamma)g_B(W+1)-(W-\rho)g_B(W)]| + P(W-\rho<0).
\end{equation}
One may bound $P(W-\rho<0)\leq\sigma^{-2}$ using Chebyshev's inequality. 
So, we now concentrate on the first term on the right--hand side of (\ref{61}).  
Throughout our proof, we will make use of the following equalities in distribution:
\begin{equation} \label{tp1}
(W|X_V=1)=_dW^s,\hspace{15pt}\mbox{ and }\hspace{15pt}(W_V|X_V=0)=_d(W|X_V=0).
\end{equation}

\vspace{2mm}
{\it Step (1)}. For this part of the proof, we will consider 
separately the cases where $\sigma^2\leq\lambda$ and $\sigma^2\geq\lambda$.  We begin
by assuming $\sigma^2\leq\lambda$, so that $\rho\geq0$.  Recall that 
\begin{equation} \label{62}
E[Wg(W)]=\lambda E[g(W^s)].
\end{equation}
Using (\ref{62}), we can then write that 
\begin{equation}\label{63}
E[(\sigma^2+\gamma)g(W+1)-(W-\rho)g(W)] = \lambda E[g(\widetilde{W})-g(W^s)],
\end{equation}
where
$$
P(\widetilde{W}=j)= \lambda^{-1}\left\{(\sigma^2+\gamma)P(W+1=j)+\rho P(W=j)\right\}, 
\quad j\geq0.
$$
That is, $\widetilde{W}=W+v_r$ where $v_r$ is a Bernoulli variable 
with success probability $r=\lambda^{-1}(\sigma^2+\gamma)$.  
Note that $r\leq1$ by assumption.  We rewrite (\ref{63}) as
\begin{equation}\label{64}
\lambda E[g(\widetilde{W})-g(W^s)] =
\lambda E[g(\widetilde{W})-g(\overline{W})] + \lambda E[g(\overline{W})-g(W^s)],
\end{equation}
by defining $\overline{W}=W_V+1$, where $V$ is a random index chosen 
according to (\ref{32}).  
For the first term in (\ref{64}) we note that, by conditioning on $v_r$,
\begin{equation} \label{tp2}
\lambda Eg(\widetilde{W}) \;=\; \lambda Eg(W+v_r) \;=\; (\sigma^2+\gamma)E\Delta g(W) + \lambda Eg(W).
\end{equation}
Furthermore, by conditioning on $X_V$ and using the equalities (\ref{tp1}),
\begin{equation} \label{tp3}
\lambda Eg(\overline{W}) \;=\;\lambda Eg(W_V+1) \;=\;\lambda_2 Eg(W^s) + (\lambda-\lambda_2) E[g(W)|X_V=0],
\end{equation}
since $P(X_V=1)=\lambda^{-1}\lambda_2$.  Again by considering conditioning on $X_V$ and using (\ref{tp1}), we have that
\begin{equation} \label{tp4}
(\lambda-\lambda_2) E[g(W)|X_V=0] \;=\; \lambda Eg(W+1) - \lambda_2 Eg(W^s+1).
\end{equation}
Combining (\ref{tp2}), (\ref{tp3}) and (\ref{tp4}) we obtain the following.
\begin{eqnarray}
\nonumber \lambda E[g(\widetilde{W})-g(\overline{W})] &=& (\sigma^2+\gamma-\lambda)E\Delta g(W) + \lambda_2E\Delta g(W^s)\\
\nonumber &=& \lambda_2E[\Delta g(W^s)-\Delta g(W)] + \gamma E\Delta g(W) \\
\label{66} &&\qquad\qquad+ (\sigma^2-\lambda+\lambda_2)E\Delta g(W).
\end{eqnarray}
Now consider the second term of (\ref{64}).  Let us combine it with the final term of 
(\ref{66}).  Since 
$$
E[\overline{W}-W^s]=-\lambda^{-1}(\sigma^2-\lambda+\lambda_2),
$$  
and proceeding as we did in deriving (\ref{3}), we get that
\begin{multline} \label{tp5}
 \lambda E[g(\overline{W})-g(W^s)] + (\sigma^2-\lambda+\lambda_2)E\Delta g(W)\\ 
 = \lambda E\sum_{j=0}^\infty\left(\Delta g(j) - \Delta g(W)\right)
\left[P(\overline{W}> j)-P(W^s> j)\right].
\end{multline}
Using the definition of $\overline{W}$, conditioning on $X_V$ and employing (\ref{tp1}), we have that
\begin{multline} \label{tp6}
\lambda\big[P(\overline{W}> j)-P(W^s> j)\big] \\
=(\lambda-\lambda_2)\big[P(W_V+1> j|X_V=0)-P(W_V+1> j|X_V=1)\big].
\end{multline}
Hence, the right--hand side of (\ref{tp5}) becomes
\begin{equation} \label{67}
(\lambda-\lambda_2) E\sum_{j=0}^\infty\left(\Delta g(j) - 
\Delta g(W)\right)\big[P(W_V+1>j|X_V=0)-P(W_V+1>j|X_V=1)\big].
\end{equation}
Let us now insert the representations (\ref{66}) and (\ref{67}) into (\ref{63}) 
and then (\ref{61}).  We obtain
\begin{multline*}
d_{TV}(\mathcal{L}(W),\mbox{TP}(\lambda,\sigma^2))
\leq (\lambda-\lambda_2)\sup_{B\subseteq\mathbb{Z}^+}\big\{\Lambda_B\big\}
+\lambda_2\sup_{B\subseteq\mathbb{Z}^+}\big|E[\Delta g_B(W^s)-\Delta g_B(W)]\big|\\
+ \gamma\sup_{B\subseteq\mathbb{Z}^+}\big|E\Delta g_B(W)\big| + P(W-\rho<0),
\end{multline*}
where
$$
\Lambda_B = E\sum_{j=0}^\infty|\Delta g_B(j)-\Delta g_B(W)||P(W_V+1>j|X_V=0)-P(W_V+1>j|X_V=1)|.
$$
Recalling that $P(W-\rho<0)\leq\sigma^{-2}$, $\gamma\leq1$ and $\|\Delta g_B\|_\infty\leq\sigma^{-2}$, 
we have that
$$
\gamma\sup_{B\subseteq\mathbb{Z}^+}\big|E\Delta g_B(W)\big| + P(W-\rho<0) \;\leq\; 2\sigma^{-2}.
$$
Furthermore, the random variable $W^s$ having the $W$--size--biased distribution
satisfies
$$
P(W^s=j) \;=\; \lambda^{-1}jP(W=j),\quad 0\leq j\leq n,
$$ 
and so,
\begin{equation}\label{71}
2d_{TV}(\mathcal{L}(W),\mathcal{L}(W^s)) \;=\; \sum_{j=0}^\infty|P(W=j)-
P(W^s=j)| \;=\; E|1-\lambda^{-1}W| \;\leq\; \lambda^{-1}\sigma.
\end{equation}
We thus have that
$$
\lambda_2\big|E[\Delta g_B(W^s)-\Delta g_B(W)]\big| \;\leq\; 2\lambda_2\|\Delta g_B\|_\infty d_{TV}(\mathcal{L}(W),\mathcal{L}(W^s)) \;\leq\; \frac{\lambda_2}{\lambda\sigma}.
$$
Combining the above bounds, we obtain
\begin{equation} \label{68}
d_{TV}(\mathcal{L}(W),\mbox{TP}(\lambda,\sigma^2)) \leq 
(\lambda-\lambda_2)\sup_{B\subseteq\mathbb{Z}^+}\big\{\Lambda_B\big\}
+ \frac{\lambda_2}{\lambda\sigma} + \frac{2}{\sigma^2}.
\end{equation}
In the second step of the proof, we consider how $\Lambda_B$ may be bounded.  
Before doing this, we show that if $\sigma^2\geq\lambda$ then the bound
(\ref{68}) continues to hold.

Consider now the case where $\sigma^2\geq\lambda$, so that $\rho\leq0$.  
We will use an analogous argument to show that the bound
(\ref{68}) continues to hold. In place of (\ref{64}), 
we this time write
\begin{multline}\label{69}
E[(\sigma^2+\gamma)g(W+1)-(W-\rho)g(W)] = (\sigma^2+\gamma) 
E[g(W+1)-g(\widehat{W})]\\
 + (\sigma^2+\gamma) E[g(\widehat{W})-g(W^\star)],
\end{multline}
where $\widehat{W}=W+v_t(1-X_V)$, $W^\star=v_tW^s+(1-v_t)W$ and 
$t=\lambda(\sigma^2+\gamma)^{-1}$.  Consider the first term on the 
right--hand side of (\ref{69}).  For this term, 
we argue as we did to derive (\ref{66}).  Conditioning on 
$v_t$ and $X_V$ and employing the equalities (\ref{tp1}), we find, as
for (\ref{66}), that
\begin{multline*}
(\sigma^2+\gamma)E[g(W+1)-g(\widehat{W})]\\
=\lambda_2E[\Delta g(W^s)-\Delta g(W)] + \gamma E\Delta g(W) 
+ (\sigma^2-\lambda+\lambda_2)E\Delta g(W).
\end{multline*}
As we have that
$$
E[\widehat{W}-W^\star]=-(\sigma^2+\gamma)^{-1}(\sigma^2-\lambda+\lambda_2),
$$
we then write
\begin{multline} \label{70}
(\sigma^2+\gamma) E[g(\widehat{W})-g(W^\star)] + 
(\sigma^2-\lambda+\lambda_2)E\Delta g(W)\\ 
= (\sigma^2+\gamma) E\sum_{j=0}^\infty\left(\Delta g(j) - 
\Delta g(W)\right)[P(\widehat{W} > j)-P(W^\star > j)].
\end{multline}
Using the definitions of $\widehat{W}$ and $W^\star$, and 
conditioning on $v_t$, we find that
$$
P(\widehat{W} > j)-P(W^\star > j) = t\left[P(\overline{W}
> j)-P(W^s > j)\right].
$$
Comparing this with (\ref{tp5}), recalling the definition of $t$ and
using (\ref{tp6}), we find that (\ref{67}) also gives us a representation
of (\ref{70}).  Continuing the argument as before, the bound (\ref{68}) 
holds too in the present case.

\vspace{2mm}
{\it Step (2)}.  In this part of the proof, we bound $\Lambda_B$,
and thus obtain the bounds of our theorems. In doing so, we will use our stochastic 
ordering and dependence assumptions. The cases where 
$X_1,\ldots,X_n$ are positively and negatively related will 
be discussed separately. In the positive related case, 
the argument of Lemma \ref{res7} shows that
$$
P(W_V+1>j|X_V=0)-P(W_V+1>j|X_V=1)\leq0, \quad j\geq 0.
$$
Noting that $(W_V+1|X_V=1)=_dW^s$, we fix some $l\in\mathbb{Z}^+$ and write
\begin{multline}\label{72}
P(W_V+1>j|X_V=1)-P(W_V+1>j|X_V=0) \\
= P(W_V+1>j+l|X_V=1)-P(W_V+1>j|X_V=0) + \sum_{i=1}^lP(W^s=j+i). 
\end{multline}
Suppose now that there is some $q\in[0,1]$ such that for each $j\geq0$
\begin{eqnarray} 
\nonumber && P(W_V+1>j+l|X_V=1)-P(W_V+1>j|X_V=0) \\
\label{73} &&\qquad\qquad\qquad\qquad\qquad \leq q\,P(W_V=j+l|X_V=1)\\
\label{tp10} &&\qquad\qquad\qquad\qquad\qquad =q\,P(W^s=j+l+1).
\end{eqnarray}
We will show presently that this is implied by the stochastic 
ordering assumption (\ref{58}).  Using (\ref{72}) and (\ref{tp10}), 
we find that
\begin{multline} \label{tp7}
\Lambda_B
\leq qE|\Delta g_B(W^s-l-1)-\Delta g_B(W)| + \sum_{i=1}^lE|\Delta g_B(W^s-i)-\Delta g_B(W)|\\
\leq 2q\|\Delta g_B\|_\infty d_{TV}(\mathcal{L}(W),\mathcal{L}(W^s-l-1))  
+ 2\|\Delta g_B\|_\infty \sum_{i=1}^ld_{TV}(\mathcal{L}(W),\mathcal{L}(W^s-i)).
\end{multline}
Using our bound on $\|\Delta g_B\|_\infty$ and the triangle 
inequality for total variation distance, the first term of (\ref{tp7})
is bounded by
\begin{multline} \label{tp8}
2q\sigma^{-2}\big\{d_{TV}(\mathcal{L}(W),
\mathcal{L}(W^s))+(l+1)d_{TV}(\mathcal{L}(W^s),\mathcal{L}(W^s+1))\big\}\\
\leq 2q\sigma^{-2}\left\{\frac{\sigma}{2\lambda}+(l+1)d_{TV}(\mathcal{L}(W^s),\mathcal{L}(W^s+1))\right\},
\end{multline}
where this last inequality uses (\ref{71}).  Similarly, the second term of
(\ref{tp7})  may be bounded by
\begin{multline} \label{tp9}
2\sigma^{-2}\sum_{i=1}^l\big\{d_{TV}(\mathcal{L}(W),\mathcal{L}(W^s))+i\,d_{TV}(\mathcal{L}(W^s),\mathcal{L}(W^s+1))\big\}\\
\leq \sigma^{-2}\left\{\frac{l\sigma}{\lambda}+l(l+1)d_{TV}(\mathcal{L}(W^s),\mathcal{L}(W^s+1))\right\}.
\end{multline}
Combining (\ref{tp7}), (\ref{tp8}) and (\ref{tp9}) with the bound (\ref{68}) 
yields the desired inequality (\ref{59}).

So, the proof of Theorem \ref{tppos} is completed upon showing 
that the stochastic ordering condition (\ref{58}) implies the inequality (\ref{73}). Writing 
$$
P(W_V=j+l|X_V=1)=P(W_V+1>j+l|X_V=1)-P(W_V>j+l|X_V=1),
$$
for $0\leq j\leq n$, it can be seen that (\ref{73}) is equivalent to
$$
P(W_V+1>j|X_V=0) \geq (1-q)P(W_V+1-l>j|X_V=1) + qP(W_V-l>j|X_V=1),
$$
for $j\geq 0$.  This, in turn, is equivalent to the stochastic ordering
\begin{equation}\label{75}
(W+1|X_V=0) \succeq_{st} (1-v_q)(W-l|X_V=1) + v_q(W-l-1|X_V=1),
\end{equation}
which can be seen using (\ref{tp1}).
Some rearranging shows that the stochastic ordering assumption 
(\ref{58}) implies the stochastic ordering (\ref{75}), hence the result 
of Theorem \ref{tppos}.

We turn our attention now to the case of negative relation, and complete
the proof of Theorem \ref{tpneg}.  When $X_1,\ldots,X_n$ are negatively related, 
one can use a similar argument to the above. We have here that
$$
P(W_V+1>j|X_V=0)-P(W_V+1>j|X_V=1)\geq0, \quad 0\leq j\leq n.
$$
Analogously to the positively related case, we write, 
for some fixed $l\in\mathbb{Z}^+$,  
\begin{multline*}
P(W_V+1>j|X_V=0)-P(W_V+1>j|X_V=1) \\
= P(W_V+1>j|X_V=0) - P(W_V+1>j-l|X_V=1) + \sum_{i=0}^{l-1}P(W^s=j-i).
\end{multline*}
This time, we suppose that there is $q\in[0,1]$ such that
\begin{equation} \label{76}
P(W_V+1>j|X_V=0)-P(W_V+1>j|X_V=1) \leq qP(W_V+1+l=j|X_V=1).
\end{equation}
Following a similar argument to that used in the case of positive
relation, we find that
\begin{equation*}
\Lambda_B \leq
\frac{l+q}{\lambda\sigma} + 
\frac{l(l+2q-1)}{\sigma^2}\,
d_{TV}(\mathcal{L}(W^s),\mathcal{L}(W^s+1)).
\end{equation*}
Combining this with (\ref{68})  gives us the desired 
inequality (\ref{57}).  It remains to show that the stochastic 
ordering assumption (\ref{56}) implies the inequality (\ref{76}), 
which can be done as above.

\section{Another abstract approximation theorem} \label{bd3}

Our aim hereafter is to consider an alternative approximation 
theorem which may be found within the present framework. 
For concreteness, we suppose that the birth rates $\alpha_j$ 
and death rates $\beta_j$ are such that the random variable 
$\pi$ has two parameters available to choose. This will be the 
case in the application presented afterwards.

Let us return to the basic representation (\ref{13}). 
To choose the two parameters of $\pi$, it seems natural, 
in our context, to consider $s=2$ and introduce the two 
conditions $\alpha=\beta$ and $EW_\alpha=EW_\beta$ 
(i.e., $E[\alpha_W(W+1)]=E[\beta_WW]$). With these choices, 
the representation (\ref{13}) becomes
\begin{equation} \label{78}
Eh(W)-Eh(\pi) = \alpha \sum_{i=0}^\infty \Delta^2f(i)\,
E[(W_\alpha-i-1)_+-(W_\beta-i-1)_+].
\end{equation}

Moreover, suppose that one can construct $W_\alpha$ and 
$W_\beta$ on the same probability space in such a way that 
$W_\beta=W_\alpha+Y$ for some random variable $Y$ which 
takes values in the set $\{-1,0,1\}$. Under this assumption, 
$E[W_\alpha]=E[W_\beta]=E[W_\alpha+Y]$, which implies $E[Y]=0$.
It is easily seen that the representation (\ref{78}) can be rewritten as
\begin{eqnarray}\label{79}
\nonumber  Eh(W)-Eh(\pi) &=& -\alpha \sum_{i=0}^\infty 
\Delta^2f(i) \; E[ YI_{(W_\alpha-1 \geq i+1)} + Y_+I_{(W_\alpha-1=i)} ]\nonumber\\
&=& -\alpha \; E[ I_{(Y=1)}\Delta^2f(W_\alpha-1) + Y\Delta f(W_\alpha-1) ].
\end{eqnarray}
Noting that 
\begin{eqnarray*}
|E[ I_{(Y=1)}\Delta^2f(W_\alpha-1) ]| &\leq& 2 \|\Delta 
f\|_\infty \; d_{TV}(\mathcal{L}(W_\alpha),\mathcal{L}
(W_\alpha+1)) \; \sup_W\{P(Y=1|W_\alpha)\},\\
|E[ Y\Delta f(W_\alpha-1) ]| \; &\leq& \; 
\|\Delta f\|_\infty E|E[Y|W_\alpha]| \;\; \leq \;\; 
\|\Delta f\|_\infty \sqrt{\textrm{Var}(E[Y|W_\alpha])},
\end{eqnarray*}
we can immediately bound the right-hand side of (\ref{79}) 
to obtain the following.

\begin{proposition} \label{pbd1}
Suppose that $\alpha=\beta$ and $EW_\alpha=EW_\beta$. If $W_\alpha$ 
and $W_\beta$ can be constructed on the same probability space such that 
\begin{equation}\label{80}
W_\beta=W_\alpha+Y \quad \mbox  {for some random variable 
$Y$ valued in $\{-1,0,1\}$},
\end{equation}
then,
\begin{eqnarray}\label{81}
|Eh(W)-Eh(\pi)|  
 \leq \;\; 2\alpha \|\Delta Sh\|_\infty\, d_{TV}(\mathcal{L}
 (W_\alpha),\mathcal{L}(W_\alpha+1))\, \sup_W\{P(Y=1|W_\alpha)\} \nonumber\\
 +\; \alpha \|\Delta Sh\|_\infty \sqrt{\mbox{Var}(E[Y|W_\alpha])}.
\end{eqnarray}
\end{proposition}

Clearly, if such a random variable $Y$ takes values on 
a bounded set other than $\{-1,0,1\}$, a representation 
analogous to (\ref{79}) may still be found,
and a result analogous to Proposition \ref{pbd1} is available. 
We now apply our Proposition \ref{pbd1} to approximate 
a sum of independent indicator random variables.

\begin{example}  
Suppose that $W=X_1+\cdots+X_n$ is the sum of independent Bernoulli 
random variables with success probabilities $p_i$, $1\leq i\leq n$. 
Brown and Xia (2001, Section 3) showed that in this case, one can 
improve on Poisson or binomial approximation
for $W$ by using a so--called polynomial birth--death 
distribution, with the choices $\alpha_j=\alpha$ and 
$\beta_j=\gamma j+j(j-1)$ for some constants 
$\alpha$ and $\gamma$. 

We will follow that approach and choose here $\alpha$ 
and $\gamma$ such that $\alpha=\beta$ and $E[\alpha_W(W+1)]=E[\beta_WW]$. 
Straightforward computations then give 
us expressions for these parameters: 
\begin{equation}\label{82}
\gamma = \lambda^2\lambda_2^{-1}-1-2\lambda+2\lambda_3\lambda_2^{-1}, 
\quad \mbox {and} \quad \alpha = \gamma\lambda+\lambda^2-\lambda_2,
\end{equation}
where $\lambda_k=\sum_{i=1}^np_i^k$ and $\lambda=\lambda_1=E[W]$ (as in Section 5). 
Note that the parameter choices (\ref{82}) are the same as those 
employed by Brown and Xia (2001), who based their selection on minimising 
the error bound obtained in their result. 

To begin with, let us prove that the condition (\ref{80}) 
is satisfied. Since the birth rate is constant 
(as in the Poisson case), we again have that $W_\alpha=W+1$. 
Let us turn our attention to $W_\beta$.  We let $W_i=W-X_i$, 
and $W_{i,j}=W-X_i-X_j$, $0\leq i,j \leq n$ and observe that 
$W(W-1)=\sum_{1\leq i\not=j\leq n}X_iX_j$. By the definition of 
$W_\beta$, we get that
\begin{eqnarray*}
P(W_\beta=k) &=& \alpha^{-1}E\left\{[\gamma W+W(W-1)]I_{(W=k)}\right\}\\
&=& \alpha^{-1}[\gamma\sum_{i=1}^np_iP(W_i+1=k)+
\sum_{1\leq i\not=j\leq n}p_ip_jP(W_{i,j}+2=k)],
\end{eqnarray*}
for $1\leq k\leq n$.  In the spirit of the size-biasing construction of Section \ref{pois}, 
we define now two random indices $T,U\in\{1,\ldots,n\}$ chosen 
according to the distribution
$$
P(T=i,U=j)=\frac{p_ip_j}{\lambda^2-\lambda_2},\quad i\neq j, \quad\quad 
\mbox{and} \quad\quad P(T=U=i)=0.
$$ 
Recall also the definition (\ref{32}) of the random index $V$.  
Combining these definitions with the above, we may write
$$
P(W_\beta=k) \;=\; \alpha^{-1}\gamma\lambda P(W+1-X_V=k) + \alpha^{-1}
(\lambda^2-\lambda_2)P(W+2-X_T-X_U=k),
$$
for $1\leq k\leq n$.  Let $q=\alpha^{-1}\gamma\lambda$; note from (\ref{82}) that 
$0\leq q\leq 1$ whenever $\gamma\geq0$.  In the sequel we will assume
that this is indeed the case. Introduce a Bernoulli random variable $v_q$
 with success probability $q$, independent of all other entries. 
We may then write
\begin{equation*}
W_\beta = v_q(W+1-X_V)+(1-v_q)(W+2-X_T-X_U) = W+1+Y = W_\alpha+Y,
\end{equation*}
where 
\begin{equation}\label{84}
Y=(1-v_q)(1-X_T-X_U)-v_qX_V,
\end{equation}
$Y$ being valued in $\{-1,0,1\}$ with $E[Y]=0$, as desired.

Now, let us evaluate the bound (\ref{81}). First, we need a bound on 
the solution $f$ of the Stein equation in this situation. 
By Theorem 2.10 of Brown and Xia (2001), one knows that
\begin{equation}\label{85}
\sup \lbrace \|\Delta Sh\|_\infty : h(j)=I_{(j\in B)},\; 
B\subseteq{\mathbb{Z}^+} \rbrace \leq \alpha^{-1}.
\end{equation} 
Further, $W$ being a sum of independent indicators, one has 
(from Barbour and Jensen (1989, Lemma 1)) 
\begin{equation}\label{86}
d_{TV}(\mathcal{L}(W),\mathcal{L}(W+1)) \;\leq\; 
\frac{1}{2\sqrt{\sum_{i=1}^n p_i(1-p_i)}}.
\end{equation}
Finally, consider the two conditional terms in (\ref{81}). 
Note from (\ref{84}) that $Y=1$ if and only if $v_q=X_T=X_U=0$, so that
\begin{eqnarray*}
P(Y=1|W) &=& (1-q)P(X_T=X_U=0|W) \;\;=\;\; (1-q)E[(1-X_T)(1-X_U)|W]\\ 
&=& \alpha^{-1} \!\! \sum_{1\leq i\not=j\leq n}p_ip_jE[(1-X_i)(1-X_j)|W].
\end{eqnarray*} 
This probability takes its greatest value when $W=0$, 
with $E[(1-X_i)(1-X_j)|W=0]=1$ for all $i$ and $j$. Hence,
\begin{equation}\label{88}
\sup_W \{P(Y=1|W)\} = \alpha^{-1} \!\! \sum_{1\leq i\not=j\leq n}p_ip_j = 
\alpha^{-1}(\lambda^2-\lambda_2).
\end{equation}      
Now, let $\|Z\|=(E[Z^2])^{1/2}$ be the $L_2$ norm for any random
variable $Z$. Since $T=_d U$ and $E[Y]=0$, we write 
$$
E[Y|W]= -q(E[X_V|W] -E[X_V]) - 2(1-q)(E[X_T|W] -E[X_T]),
$$
and thus 
\begin{eqnarray*}
\sqrt{\mbox{Var}(E[Y|W])} &=& \|E[Y|W]\|\\
&\leq & q\sum_{j=1}^n \|E[X_j|W] -E[X_j]\| P(V=j)\\ 
&&\qquad + 2(1-q) \sum_{j=1}^n \|E[X_j|W] -E[X_j]\| P(T=j) \\
&\leq& (q+2(1-q)) \max_{1\leq j\leq n} \sqrt{\mbox{Var}(E[X_j|W])}.
\end{eqnarray*}
When $p_j=p$ for $j=1,\ldots,n$, $E[X_j|W]=W/n$ and so the bound becomes the 
equality
\begin{equation}\label{87}
\sqrt{\mbox{Var}(E[Y|W])} = (2-q) \sqrt{\mbox{Var}(W/n)}.
\end{equation}
Inserting (\ref{85}), (\ref{86}), (\ref{88}) and (\ref{87}) in (\ref{81}) 
then provides the following bound:
\begin{equation*}
d_{TV}(\mathcal{L}(W),\mathcal{L}(\pi)) \;\leq\; 
\frac{p}{(1-p)\sigma} +\frac{(2-q)\sigma}{n} =O(p/\sqrt{\lambda}),
\end{equation*}
where $\sigma^2=\mbox{Var}(W)=np(1-p)$.

By exploring the explicit structure of the auxiliary variable $Y$, it is possible to 
derive better bounds. Throughout this part we let $\bar{a}=1-a$ for any $a\in\mathbb{R}$
and $\sigma_k=\sqrt{\sum_{i=k+1}^n\rho_i}$, where $\rho_i$ is the $i$th largest number of $p_1(1-p_1),\ldots,p_n(1-p_n)$.  From Barbour and Jensen (1989, Lemma 1) we have that for all 
$i,j=1,\ldots,n$ and $i\not=j$,
$$
2d_{TV}(\mathcal{L}(W_i),\mathcal{L}(W_i+1))\leq\sigma_1^{-1} \quad\mbox{and}\quad 2d_{TV}(\mathcal{L}(W_{i,j}),\mathcal{L}(W_{i,j}+1))\leq\sigma_2^{-1}.
$$
Notice that, from representation (\ref{84}),
\begin{equation}\label{90}
I_{(Y=1)}=\bar{v}_q \bar{X}_T\bar{X_U}\;,\; 
I_{(Y=-1)}=v_q X_V+\bar{v}_q X_T X_U.\end{equation}
The derivations below are based on the conditional independence of
$X_T$ and $W_T$, given $T$ and similarly $X_U$ and $W_U$, given $U$
and $X_V$ and $W_V$, given $V$.  By substituting (\ref{90}) in (\ref{79}), integrating with respect to $v_q$,  
separating linear and quadratic terms and noticing that $T=_d U$,  
we derive, after some simple calculations,
\begin{eqnarray*}
I&=& Eh(W)-Eh(\pi)\\
&=& -\alpha E[\bar{v}_q \bar{X}_T\bar{X}_U \Delta f(W+1)] + 
\alpha E[(v_q X_V  + \bar{v}_q X_T X_U)\Delta f(W)] \\
&=&-\alpha \bar{q} E[X_T X_U \Delta^2 f(W)] \\
&&\quad\quad +2\alpha \bar{q} E[X_T\Delta^2 f(W)]\\
&&\quad\quad - \alpha( \bar{q} E[\Delta f(W+1)]  - 
E[(2\bar{q}X_T + q X_V)\Delta f(W)])\\
&=&I_1+I_2+I_3.
\end{eqnarray*}
Using the conditional independence of 
$W_{T,U}$ and $X_T, X_U$ given $T$ and $U$, the first term $I_1$ is bounded by
\begin{eqnarray*}
|I_1|&=& \alpha\bar{q}\big|EE[X_T X_U|T,U] E[ \Delta^2 f(W_{T,U}+2)]\big|\\
&\leq& 2\alpha\bar{q}\|\Delta f\|_\infty E[X_T X_U]  \max_{i\neq j}\big\{d_{TV}(\mathcal{L}(W_{i,j}),\mathcal{L}(W_{i,j}+1))\big\}
\;\leq\; \frac{\lambda_2^2-\lambda_4}{\alpha \sigma_2}.
\end{eqnarray*}
By conditioning on $T$,
\begin{eqnarray*}
|I_2|&=& 2\alpha \bar{q}\big|EE[X_T|T] E[ \Delta^2 f(W_{T}+1)]\big|\\
&\leq& 4\alpha\bar{q}\|\Delta f\|_\infty E[X_T] \max_{i}\big\{d_{TV}(\mathcal{L}(W_{i}),\mathcal{L}(W_{i}+1))\big\}
\;\leq\; \frac{2(\lambda \lambda_2-\lambda_3)}{\alpha \sigma_1}.
\end{eqnarray*}
To bound $I_3$, we first notice that since $E[Y]=0$, 
$$
\bar{q}=2\bar{q}E[X_T] + q E[X_V].
$$
Thus,
\begin{eqnarray*}
|I_3| &=& \big|2\alpha\bar{q} E\big\{X_T \big(E[\Delta f(W_T+1)|T] - 
E[\Delta f(W_T+X_T+1)]\big)\big\}\\
&&\quad + \alpha q E\big\{X_V \big(E[\Delta f(W_V+1)|V]  
- E[\Delta f(W_V+X_V+1)]\big)\big\}\big|\\
&\leq& 
2\alpha \big\{2\bar{q} E[\{E(X_T|T)\}^2] \\
&&\qquad\qquad\qquad + qE[\{E(X_V|V)\}^2]\big\} \|\Delta f\|_\infty 
\max_{i}\big\{d_{TV}(\mathcal{L}(W_{i}),\mathcal{L}(W_{i}+1))\big\}\\
&\leq&\frac{2(\lambda\lambda_3-\lambda_4)}{\alpha\sigma_1}+\frac{\gamma\lambda_3}{\alpha\sigma_1}.
\end{eqnarray*}
By combining the bounds on $I_1,I_2$ and $I_3$ we derive the following.
\begin{proposition} \label{cor_ex8}
With $W$ and $\pi$ as above,
\begin{equation} \label{newbound}
d_{TV}(\mathcal{L}(W),\mathcal{L}(\pi)) \;\leq\; 
\frac{\lambda_2^2-\lambda_4}{\alpha \sigma_2}
+\frac{2(\lambda \lambda_2-\lambda_3)}{\alpha \sigma_1}
+\frac{2(\lambda\lambda_3-\lambda_4)}{\alpha\sigma_1}+\frac{\gamma\lambda_3}{\alpha\sigma_1}.
\end{equation}
\end{proposition}

Let us conclude by comparing our result with that of Brown and Xia (2001, Theorem 3.1), 
who obtain
\begin{equation} \label{bxbound}
d_{TV}(\mathcal{L}(W),\mathcal{L}(\pi)) \;\leq\; \frac{\gamma\lambda_3}{\alpha\sigma_1} + \frac{2\lambda\lambda_2}{\alpha\sigma_2}.
\end{equation}
When $p_i= p\to 0$ for each $i$ and $\lambda\to \infty$, both the bounds (\ref{newbound}) and 
(\ref{bxbound}) are asymptotically equivalent to $3p^2/\sqrt{\lambda}$.
\end{example}

\acks

Fraser Daly gratefully acknowledges the financial support of the Schweizerische Nationalfonds, and thanks 
the Belgian Fonds National de la Recherche Scientifique for support during a visit to the Universit\'e
Libre de Bruxelles. Thanks are also due to Andrew Barbour for several useful discussions.

\end{document}